\documentclass{amsart}
\usepackage{amssymb,amsmath,amsthm,amsfonts,amsopn,url}
\usepackage{enumitem}
\theoremstyle{plain}
\newtheorem{thm}{Theorem}[section]

\newtheorem{lemma}[thm]{Lemma}
\newtheorem{cor}[thm]{Corollary}
\theoremstyle{definition}
\newtheorem{defn}[thm]{Definition}
\newtheorem*{defn*}{Definition}
\newtheorem{example}[thm]{Example}

\theoremstyle{remark}
\newtheorem{rmk}[thm]{Remark}
\newtheorem*{rmk*}{Remark}

\newcommand{\N}{\field{N}}
\newcommand{\Z}{\field{Z}}
\newcommand{\Q}{\field{Q}}
\newcommand{\R}{\field{R}}
\newcommand{\C}{\field{C}}
\newcommand{\ideal}[1]{\mathfrak{#1}}
\newcommand{\m}{\ideal{m}}
\newcommand{\n}{\ideal{n}}

\newcommand  {\shE}     {\mathcal{E}}
\newcommand  {\shF}     {\mathcal{F}}
\newcommand  {\shG}     {\mathcal{G}}

\newcommand  {\shM}     {\mathcal{M}}
\newcommand  {\shN}     {\mathcal{N}}
\newcommand  {\shL}     {\mathcal{L}}

\newcommand  {\shQ}     {\mathcal{Q}}

\newcommand  {\shS}     {\mathcal{S}}

\newcommand  {\dual}    {\vee}

\newcommand  {\im}      {\operatorname{im}}

\renewcommand  {\ker }  {\operatorname{kern}}
\newcommand  {\koker }  {\operatorname{cokern}}

\renewcommand{\O}       {\mathcal{O}}

\newcommand  {\Proj}    {\operatorname{Proj}}

\newcommand  {\Spec}    {\operatorname{Spec}}

\newcommand  {\Syz}     {\operatorname{Syz}}

\newcommand{\oplusdots}{ \oplus \ldots \oplus }
\newcommand{\timesdots}{ \times \ldots \times }

\renewcommand{\N}{{\mathbb N}}
\renewcommand{\Z}{{\mathbb Z}}
\renewcommand{\Q}{{\mathbb Q}}
\renewcommand{\R}{{\mathbb R}}
\renewcommand{\C}{{\mathbb C}}

\newcommand{\nesenum}{{\operatorname{N}^1}}

\newcommand{\hkmodulefct}{\operatorname{HKF}}

\newcommand{\hikmodulefct}{\operatorname{HKF^{i}}}

\newcommand{\hkmodulemult}{\operatorname{HK}}

\newcommand{\hikmodulemult}{\operatorname{HK^{i}}}

\newcommand{\at}{a}

\newcommand{\aat}{b}

\newcommand{\aet}{c}

\newcommand{\surface}{{S}}

\author{Holger Brenner}

\title{Irrational Hilbert-Kunz multiplicities}

\date{\today}
\begin{document}

\large

\begin{abstract}
We interpret Hilbert-Kunz theory of a graded ring of positive characteristic in terms of Frobenius asymptotic of cohomology
of vector bundles on projective varieties.
With this method we show that for almost all prime numbers there exist three-dimensional quartic hypersurface domains and artinian
modules with
irrational Hilbert-Kunz multiplicity. From this we deduce that also
the Hilbert-Kunz multiplicity of a local noetherian domain
might be an irrational number.
\end{abstract}



\maketitle

\section*{introduction}

Let $(R,\m)$ be a local noetherian ring containing a field $K$ of positive characteristic $p$. For an ideal $I$ and a prime power $q=p^e$
we define $I^{[q]}= (f^q, f \in I)$ to be the extended ideal under the $e$th iteration of the Frobenius homomorphism.
If $I$ is primary to the maximal ideal $\m$ (i.e. they have the same radical), then also $I^{[q]}$ has this property and so
$R/I^{[q]}$ is supported on $\m$ and has finite length. E. Kunz first studied in \cite{kunzhilbertkunz} the function
\[ e \longmapsto   \lg (R/I^{[q]}) \]
and observed in examples that it grows of order $c q^{\dim R}$.
In \cite{monskyexists}, P. Monsky proved in general that the limit
\[ e_{HK}(I) =  \lim_{e \rightarrow \infty}  \frac{ \lg( R/ I^{[q]}   ) }{q^ {\dim R} } \]
exists as a real number; it is called the \emph{Hilbert-Kunz multiplicity} of $I$,
and the Hilbert-Kunz multiplicity of $\m$ is also
called the \emph{Hilbert-Kunz multiplicity} of $R$.
In that paper, Monsky suspected that the Hilbert Kunz multiplicity of an $\m$-primary ideal is always a rational number.
This problem has been a driving force in Hilbert-Kunz theory ever since.
Positive results on this problem
(and other questions) were obtained by several authors in many specific situations and with very different methids:
for Fermat type equations  (\cite{hanmonsky}, \cite{gesselmonsky}), cubic equations in three variables (\cite{buchweitzchen}),
binomial equations (\cite{conca}),
monoid rings (\cite{watanabetoric}),
invariant rings (\cite{watanabeyoshida2}),
rings of finite (Frobenius-)Cohen-Macaulay type (\cite{seibertfinitetype}),
two-dimensional graded rings (\cite{brennerrational}, \cite{trivedirational}).
However, Monsky conjectures in \cite{monskyirrational} that the Hilbert-Kunz multiplicity of
\[  {\mathbb Z}/(2)[X,Y,Z,U,V]/( UV+X^3+Y^3+XYZ )\]
is $ \frac{4}{3}  + \frac{5}{14 \sqrt{7} }$, hence irrational. This conjecture is open, though it is
strongly supported by computer computations.

In this paper we give an explicit example of a homogeneous hypersurface ring of degree four and of dimension three
and with an isolated singularity
such that there exists a (not so explicit) module of finite length $M$ such that the limit
$   \lim_{e \rightarrow \infty} \frac{  \lg \left( F^{e*} M \right)  }{q^3}$ 
is an irrational number (Theorem \ref{moduleartinianirrational}).
From this we deduce that there exists also a local noetherian domain whose Hilbert-Kunz multiplicity is irrational
(Theorem \ref{hilbertkunzirrational}).
These results are obtained by combining geometric, cohomological and algebraic methods.

This paper extends the study of Hilbert-Kunz theory in the graded case with the help of
vector bundles on the corresponding projective
variety to higher dimensions (meaning ring dimension $\geq 3$ and dimension of the projective variety $\geq 2$).
This method was initiated in \cite{brennertightprojective}, \cite{brennerslope}, \cite{brennerelliptic},
\cite{brennertightplus} in the context of tight closure theory and in
\cite{brennerrational} and independently in \cite{trivedirational}
for Hilbert-Kunz theory with the focus on the case of two-dimensional rings.
This approach using semistability properties of syzygy bundles on the corresponding projective curves
settled most problems in Hilbert-Kunz theory for graded normal two-dimensional rings: rationality,
limit behavior for $p \rightarrow \infty$ (\cite{trivedimodp}),
boundedness and periodicity of the constant term (\cite{brennerfunction}), relation to solid closure (\cite{brennersolid}),
and allowed also results for non-normal rings (\cite{trivedisingular}, \cite{monskyirreducible}, \cite{monskynodal}).

The basic observation of this approach is that one may express the Hilbert-Kunz multiplicity of a primary homogeneous ideal
(or a graded module of finite length) in a standard-graded domain $R$ over an algebraically closed field $K$ of positive characteristic $p$
with a formula involving the limit
\[ \lim_{e \rightarrow \infty}   \frac{ \sum_{m=0}^\infty \dim_K (H^d( Y, (F^{e*} \shS) (m))) }{p^{e (d +1) } }  \]
and simpler terms, where $Y=\Proj R$ is the corresponding projective variety of dimension $d$,
$\shS$ is the top (dimensional) syzygy bundle and $F^{e*}$ denotes the Frobenius pull-back
(Theorem \ref{hilbertkunzgeometric}).
For $d=1$, this term is controlled by the strong Harder-Narasimhan filtration of $\shS$ (see \cite[Theorem 2.7]{langer}).
The main difference in higher dimension
is that with stability conditions one can only control the zeroth and the top-dimensional sheaf cohomology of a vector bundle,
but not the intermediate cohomology.
This problem is apparent already for line bundles, independent of the question whether they can occur as a direct summand of a top
syzygy bundle or not.
Therefore it is natural to study first for line bundles how the dimensions of the cohomology groups of the twists
of their Frobenius pull-backs behave.

As Frobenius pull-backs for line bundles are just ordinary powers,
it turns out that we are dealing with a \-- to some extent \--
characteristic-free situation and that an extremely useful tool is available:
the interpretation of the intersection behavior of line bundles in terms of the numerical N\'{e}ron-Severi group $\nesenum (Y)$,
its linearization  $\nesenum (Y) \otimes_\Z \R$
and the convex geometry of the intrinsically defined ample, positive, (pseudo)effective and similar cones.
This approach was developed by S. L. Kleiman in
\cite{kleimannumerical} (following ideas of A. Grothendieck and D. Mumford \cite{mumfordcurve})
and is now ubiquitous in algebraic geometry, in particular in the minimal model program
(\cite{kollarrational}, \cite{kollarmori}, \cite{lazarsfeldpositivity1}, \cite{lazarsfeldpositivity2}).

The possible irrational boundaries of these cones were used by S. D. Cut\-kos\-ky in \cite[Example 1.6]{cutkoskyzariski}
to give an example of a divisor on a threefold such that no pull-back of it under any birational transformation has a
Zariski decomposition. He later applied this method also to problems coming from commutative algebra.
In \cite{cutkoskyirrational}, he gave an example showing that the Castelnuovo-Mumford regularity of powers of an ideal sheaf
can have an irrational limit,
and it was his talk about that paper at MSRI in March 2013 which was a starting point for the current work.
There are several sources in algebraic geometry where irrational boundaries occur. The
paper \cite{cutkoskyirrational} builds on \cite{morrisonk3} where it was shown that
any lattice together with a quadratic form fulfilling certain natural necessary conditions can be realized as
the N\'{e}ron-Severi group of an algebraic $K3$ surface over the complex numbers $\C$. 

In this paper we will also use $K3$ surfaces having a certain intersection behavior.
In his recent paper  \cite{oguisoentropy}, K. Oguiso produced
an example of a $K3$ surface $\surface$ over $\mathbb C$
where the automorphism group is large in the sense that there exists a fixpoint free automorphism
such that the corresponding homomorphism on the second singular cohomology $H^2 (\surface, {\mathbb C})$
has an eigenvalue whose absolute value is larger than $1$.
In \cite{festigarbagnatigeemenluijk}, the authors established a relationship between this example with
work of Cayley \cite{cayleymemoir} and reinterpreted it in terms of a determinantal equation given by a $4 \times 4$-matrix
whose entries are linear polynomials in four variables.

Here we will look in general at homogeneous quartic polynomials in four variables (mostly over $\Z$)
given as the determinant of a $4\times 4$-matrix with linear entries and defining a smooth projective surface $\surface$.
By work of A. Beauville \cite{beauvilledeterminantal}, we know that these surfaces have Picard rank at least two 
and we know how the intersection form on the subgroup
given by the very ample class and a certain determinantal curve looks like. Concrete examples where
the Picard rank is in fact two were established in \cite{festigarbagnatigeemenluijk} in characteristic two and in characteristic zero.
It follows from the shape of the intersection form
and the existence of the Cayley-Oguiso automorphism
that for almost all prime numbers
the ample cone of such a surface equals
the effective cone (up to closure and restricted to the plane spanned by the two mentioned divisors)
and that it has irrational boundaries (Lemma \ref{ampleeffectivesublattice}).
From this we deduce the existence of line bundles $\shL$ on $\surface$
such that the limits
\[ \lim_{e \rightarrow \infty} \frac{\sum_{m \in \N} \dim_K (H^2(\surface, \shL^{p^e}   (m) ) }{p^{3e} } \, \text{ and } \,
\lim_{e \rightarrow \infty} \frac{\sum_{m \in \Z} \dim_K (H^1(\surface, \shL^{p^e}   (m) ) }{p^{3e} }  \]
are irrational square roots of rational numbers (Corollary \ref{ampleeffectivepositive} and Corollary \ref{oguisoirrational}).

However, this does not give directly an example of irrational Hilbert-Kunz multiplicity for a primary ideal in a local ring, 
since it is not clear whether one can realize the line bundle as a direct summand of a top syzygy bundle of such an ideal.
Instead, we translate the line bundle back to a graded module over the homogeneous coordinate ring $R$.
This is an invertible module on the punctured spectrum $\Spec R \setminus \{R_+\}$ and not at all artinian.
From the geometric results we get that the
second local cohomology (with support in the maximal ideal) of its Frobenius pull-back has an irrational limit
(Corollary \ref{module2irrational}).
Now these local cohomological variants of Hilbert-Kunz theory were recently studied by H. Dao and I. Smirnov in \cite{daosmirnov},
and their results allow us to deduce from the mentioned irrational behavior the existence of a (non-primary) ideal whose
zeroth local cohomology has irrational Frobenius asymptotic (Corollary \ref{module0irrational})
and the existence of a module of finite length with irrational Hilbert-Kunz multiplicity (Theorem \ref{moduleartinianirrational}).

In the final step we show how to construct starting from a module of finite length with irrational Hilbert-Kunz multiplicity
also a local ring with irrational Hilbert-Kunz multiplicity (Theorem \ref{reductionideal}, Theorem \ref{reductionmaximal},
Theorem \ref{hilbertkunzirrational}).

We give a quick overview of the organization of this paper. In Section \ref{hilbertkunzmodule} we recall the different notions
of Hilbert-Kunz multiplicities including the recent local cohomological variants introduced in \cite{daosmirnov}. In
Section \ref{geometric} we express the Hilbert-Kunz function for a graded module of finite length
over a standard-graded domain with the top cohomology of Frobenius pull-backs and its twists of a
top syzygy bundle (Theorem \ref{hilbertkunzgeometric}).
These syzygy bundles do not necessarily have to stem from a minimal graded
resolution of the module, the requirement is just that the complex is exact on the punctured spectrum.
In Section \ref{resolution} we show for hypersurfaces in ${\mathbb P}^3$
that there are such resolutions with small ranks.

As the Frobenius-asymptotical cohomological behavior of the syzygy bundles is still
quite difficult to control, we focus further on the case of line bundles. In Section \ref{amplethreshold}
we introduce the ample and the antiample threshold of a line bundle on a smooth projective variety and compute
in Lemma \ref{antiamplethreshold} the relevant asymptotic of an antiample line bundle
in terms of its antiample threshold under the condition that the effective cone equals the ample cone up to closure. 
In Corollary \ref{ampleeffectivepositive} we specify this  showing that for a projective surface with the property that this
cone has irrational boundaries we will get an irrational Frobenius-asymptotical behavior of second cohomology of a line bundle.
In Section \ref{splitting} we apply the previous results in the case where the top syzygy bundle splits into line bundles
(Theorem \ref{hilbertkunzsplitting});
apart from the case where the module has finite projective dimension we can establish this behavior for all graded artinian
modules over quadrics in four variables and prove the (known) rationality of the Hilbert-Kunz multiplicity in Corollary
\ref{quadricrational}.

The following sections are devoted to establishing examples with irrational Hilbert-Kunz multiplicities. In
Section \ref{determinantalquartic}
we study, based on \cite{beauvilledeterminantal}, \cite{oguisoentropy} and \cite{festigarbagnatigeemenluijk}, determinantal quartics
and show in Lemma \ref{ampleeffectivesublattice} and 
Corollary \ref{oguisoirrational} that there are indeed examples where the looked-for irrational behavior for line bundles occurs.
These results are translated in Section \ref{interpretation} back to commutative rings. In
Corollary \ref{module2irrational} we establish that the second local cohomological Hilbert-Kunz multiplicity might be irrational;
from this we deduce using results of \cite{daosmirnov} that also the zeroth local cohomological Hilbert-Kunz multiplicity 
(Corollary \ref{module0irrational}) and also the Hilbert-Kunz multiplicity of a module of finite length
(Theorem \ref{moduleartinianirrational})
might be irrational. In the final Section \ref{reductions} we show independent of previous results how
one can construct starting from an artinian module with irrational Hilbert-Kunz multiplicity first a local ring with a primary ideal
having  irrational Hilbert-Kunz multiplicity (Theorem \ref{reductionideal}) and how to construct from this a local ring
with  irrational Hilbert-Kunz multiplicity (Theorem \ref{reductionmaximal}).
Theorem \ref{hilbertkunzirrational} finally gives the existence of irrational Hilbert-Kunz multiplicities.

I thank D. Cutkosky for his inspiring talk about \cite{cutkoskyirrational} and a subsequent conversation, H. Dao for
indicating important reduction steps using results from \cite{daosmirnov} and  R. van Luijk for explaining parts of
\cite{festigarbagnatigeemenluijk} to me. I thank D. Brinkmann and A. St\"abler for careful reading and D. Brinkmann
for computations with \cite{M2}. Moreover, I thank L. Avramov, O. Baranouskaya, M. Blickle, B. Brenner, D. Brink\-mann, R. Buchweitz,
D. Cutkosky, H. Dao, L. Ein,
C. Favre, H. Fischbacher-Weitz,
R. Hartshorne, M. Katzman, J. Li, G. Lyubeznik, C. Miller, R. Mir\'{o}-Roig, P. Monsky, A. St\"abler, K. Schwede, A. Singh,
I. Smirnov,
S. Takagi, B. Teissier, P. Teixeira, V. Trivedi, K. Tucker, R. van Luijk,
K. Watanabe, W. Zhang for various discussions, remarks, their interest and encouragement.

This material is based upon work supported by the National Science 
Foundation under Grant No. 0932078 000, while the author was in 
residence at the Mathematical Science Research Institute (MSRI) in 
Berkeley, California, during the spring semester 2013 in the special year in commutative algebra 2012-2013.
I thank MSRI for its hospitality during my stay.

\section{Hilbert-Kunz function for modules}
\label{hilbertkunzmodule}

For an $R$-module $M$ we denote by $F^{e*}M= M \otimes_R {}^{e} R$ the pull-back of $M$
(often called the \emph{Peskine-Szpiro functor}, see \cite[D\'{e}finition 1.2]{peskineszpiro})
under the $e$th iteration of the Frobenius homomorphism $F: R \rightarrow R$, $f\mapsto f^{p}$.
For a submodule $N \subseteq M$ we get induced $R$-module homomorphisms
\[ F^{e*}N   \longrightarrow   F^{e*}M    \, , \]
which are in general not injective anymore. The image under these homomorphisms is denoted by $N^{[q]}$, where $q=p^e$.

\begin{defn}
Let $R$ denote a noetherian commutative ring of positive characteristic $p$
with a fixed maximal ideal\footnote{Usually $(R,\m)$ will either be a local ring or a standard-graded ring $R$ with $\m=R_+$.
In general $\dim R$ should be understood as $\operatorname{ht} (\m)$.} $\m$ and let $N \subseteq M$ be finitely generated $R$-modules such that $M/N$ has support on $\m$.
Then we call
\[
\hkmodulefct (N,M,e) := \lg ( F^{e*} M   /  \im ( F^{e*} N  \longrightarrow F^{e*} M  ) ) = \lg (F^{e*}M/ N^{[q]})
 \]
the \emph{Hilbert-Kunz function} of the submodule $N \subseteq M$. The limit, 
\[ \lim_{e \rightarrow \infty}  \frac{  \lg ( F^{e*} M   /  \im ( F^{e*} N  \longrightarrow F^{e*} M ) ) }{p^{e \dim R } } \]
is called the \emph{Hilbert-Kunz multiplicity} of $N \subseteq M$ and denoted by $\hkmodulemult (N,M)$, provided that it exists.
\end{defn}

\begin{rmk}
The support condition ensures  that the support of  the modules
$F^{e*} M   /  \im ( F^{e*} N  \rightarrow F^{e*} M )$ is on the maximal ideal and hence their lengths are finite,
because they are finitely generated. So this Hilbert-Kunz function is well defined
with values in $\mathbb N$. Note that for $M=R$ and $I=N$ an ideal we have $F^{e*} (R/I) = R/I^{[q]}$, so this definition includes the
classical case of a primary ideal. If $R$ contains an algebraically closed field $K$ then the length of a module is just
its $K$-dimension.
\end{rmk}

\begin{lemma}
\label{modulerepresentation}
Let $R$ denote a noetherian commutative ring of positive characteristic $p$
with a fixed maximal ideal $\m$ and let $N \subseteq M$ be finitely generated $R$-modules such that $M/N$ has support in $\m$.
Let
\[ R^n \stackrel{A}{\longrightarrow}   R^m \longrightarrow  L=  M/N  \longrightarrow  0 \]
be an exact complex and let $\tilde{N} \subseteq R^m $ be the image of the left homomorphism. Then 
\[      \hkmodulefct (N,M,e) =  \hkmodulefct (\tilde{N},R^m,e) = \hkmodulefct (0,L,e)  \, .   \]
\end{lemma}
\begin{proof}
By the right exactness of the tensor product we get exact complexes
\[   F^{e*} N \longrightarrow    F^{e*} M   \longrightarrow    F^{e*} L \longrightarrow 0 \]
and
\[ R^n \stackrel{A^{[q]} }{\longrightarrow}   R^m \longrightarrow  F^{e*} L   \longrightarrow  0\, ,\]
where in $A^{[q]}$ every entry of $A$ is raised to the $q$th power.
Hence
\[  F^{e*} L \cong      F^{e*} M / N^{[q]}    \cong   R^m/ \tilde{N}^{[q]} \, . \]
\end{proof}

By the preceding lemma it is enough to compute the Hilbert-Kunz function and multiplicity of the
$0$-submodule inside a finitely generated module $L$ with support on $\m$
(equivalently, a module of finite length or an artinian finitely generated module).
We will denote these just by $\operatorname{HKF} (L,e)$ and  $\operatorname{HK} (L)$.

\begin{defn}
Let $R$ denote a noetherian commutative ring of positive characteristic $p$
with a fixed maximal ideal $\m$ and let $ M$ be a finitely generated $R$-module.
Then we call
\[
\hikmodulefct (M,e) := \lg ( H^{i}_\m  ( F^{e*} M   ) )
 \]
the $i$th \emph{supported} (or \emph{local cohomological}) \emph{Hilbert-Kunz function} of the module $M$,
provided that the lengths are finite. The limit, 
\[ \lim_{e \rightarrow \infty}  \frac{      \lg ( H^{i}_\m  ( F^{e*} M   ) )      }{p^{e \dim R } } \]
is called the $i$th \emph{supported Hilbert-Kunz multiplicity} of $M$ and denoted by $\hikmodulemult (M)$,
provided that it exists.
\end{defn}

\begin{rmk}
We have the  exact sequence relating local and global cohomology (setting $U= \Spec R \setminus \{\m\}$,
we denote the sheafification of a module by the same name)  
\[  0 \longrightarrow   H^0_\m (F^{e*} M)  \longrightarrow F^{e*} M \longrightarrow H^0(U,F^{e*} M ) 
\longrightarrow H^1_\m (F^{e*} M) \longrightarrow 0     \]
and isomorphisms
\[ H^{i}_\m (F^{e*} M)  \cong H^{i-1}(U,  F^{e*} M )   \]
for $i \geq 2$. The support of $H^i_\m(F^{e*}(M))$ is the maximal ideal $\m$.
This is finitely generated under certain conditions, see \cite[Section 9]{brodmannsharp}.
If $M$ is a finitely generated $R$-module with support in $\m$, then we get from the above sequence
and observing that $ M|_U=0$ the isomorphisms
\[ H^0_\m (F^{e*} M) \cong F^{e*} M  \]
and $H^{i}_\m(F^{e*} M ) = 0$ for $i \geq 1$.  
So in this case this limit is just a reformulation of the Hilbert-Kunz multiplicity of $M$.
\end{rmk}

\begin{lemma}
\label{relatingsequence}
Let $R$ denote a noetherian commutative ring of positive characteristic $p$
with a fixed maximal ideal $\m$ and let $U=\Spec R \setminus \{\m\}$ be the punctured spectrum.
Let $N \subseteq M$ be finitely generated $R$-modules.
Then there exists a short exact sequence
\[  0 \longrightarrow H^0_\m (F^{e*}M) / H^0_\m( N^{[q]}) 
\longrightarrow F^{e*}M/ N^{[q]} \]
\[\longrightarrow H^0(U,   F^{e*}M) /\im \left( F^{e*}N \longrightarrow  H^0(U,   F^{e*}M)  \right)
\longrightarrow H^1_\m (   F^{e*}M)  
\longrightarrow 0 \, .   \]
\end{lemma}
\begin{proof}
We have the following commutative diagram with exact rows 
\[
\begin{matrix}
0 & \! \rightarrow \!&  H^0_\m (F^{e*} N) & \! \rightarrow \! &  F^{e*} N & \! \rightarrow \! & H^0 (U, F^{e*} N) &
 \! \rightarrow \!  & H^1_\m (F^{e*} N)   &  \! \rightarrow \! \, \,  0  \\
& & \downarrow & & \downarrow & &  \downarrow & & \downarrow &  \\
 0 &  \! \rightarrow \! &  H^0_\m (F^{e*} M) &   \! \rightarrow \! &  F^{e*} M & \! \rightarrow \! &  H^0 (U, F^{e*} M) &
 \! \rightarrow \!  & H^1_\m (F^{e*} M)   &  \! \rightarrow \! \, \,  0  \\
 & & &  & \downarrow & & & &&  \\
 & & &  &   \! \!   F^{e*} M /  N^{[q]}    \!\!     & & & && 
 \\ & & &  & \downarrow & & & &&  \\
  & & &  &  0 \, . \! \!& && & & 
 \end{matrix}
\]
We have (inside $ F^{e*}M$)
\[  H^0_\m(N^{[q]})  =  N^{[q]} \cap H^0_\m(F^{e*} M ) \, ,  \]
hence we get an injection
\[   H^0_\m (F^{e*} M) / H^0_\m ( N^{[q]}  )   \subseteq   F^{e*} M / N^{[q]} \, .  \]
We have a homomorphism (coming from restricting to $U$)
\[ \varphi:  F^{e*} M / N^{[q]}   \longrightarrow  H^0(U, F^{e*}M) / \im(N^{[q]} )  \]
which sends $H^0_\m (F^{e*} M) / H^0_\m ( N^{[q]}  )$ to $0$.
Let $x \in  F^{e*} M / N^{[q]}$ (represented by $x \in  F^{e*} M$)
be an element mapped to $0$ under $\varphi$. Then there exists $y \in N^{[q]}$
such that $\varphi(y)=\varphi(x)$ in $H^0(U, F^{e*}M)$.
The difference $x-y  \in  F^{e*} M    $ is then mapped to $0$ in $H^0(U, F^{e*}M)$ and hence 
$x-y \in H^0_\m(F^{e*}M)$. The class $[x-y] \in H^0_\m (F^{e*}M) / H^0_\m( N^{[q]}) $ is mapped to $x$ in
$F^{e*} M / N^{[q]}$, proving the exactness at the second spot.
Finally, $N^{[q]}$ is mapped to $0$ in $H^1_\m (F^{e* }M)$, therefore
we get a surjection
\[    H^0(U, F^{e*}M) / \im(N^{[q]} )    \longrightarrow  H^1_\m (F^{e* }M) \, ,  \]
which sends $ F^{e*}M/ N^{[q]}$ to $0$. The exactness at the third spot is clear from the second row of the diagram.
\end{proof}

\begin{rmk}
Note that the given sequence is not the long exact sequence attached to $F^{e*} (M/N)$.
The modules $H^0_\m(F^{e*}M) / H^0_\m(N^{[q]})$
are the kernels of $H^0_\m(F^{e*}M)  \rightarrow  H^0_\m(F^{e*}(M/N) )  $.
So their lengths are not given by a supported Hilbert-Kunz function, but by another more general construction. Even for $N=M$
this sequence is not trivial,
it degenerates to $0\rightarrow H^0(U,F^{e*}M)/ \rho_U(F^{e*}M)   \rightarrow   H^1_\m (F^{e*}M) \rightarrow 0$ (where $\rho_U$ denotes the
restriction homomorphism to $U$).
\end{rmk}

\begin{rmk}
The Hilbert-Kunz function $e \mapsto R/\m^{[p^{e}]}$ of the maximal ideal was introduced by E. Kunz in \cite{kunzhilbertkunz}.
P. Monsky proved in \cite{monskyexists}
the existence of the limit $\lim_{e \rightarrow \infty} \frac{ \lg \left( M/MI^{[p^{e} ]} \right)}{p^{e \dim R}}$
for a finitely generated $R$-module $M$ and an $\m$-primary ideal $I$.
At around the same time S. P. Dutta considered in \cite[Corollary 2, Lemma 1.6]{duttafrobenius}
the length of $F^{e*}M/p^{e \dim R}$ for an $R$-module $M$ of finite length, in particular in the case of
finite projective dimension; G. Seibert unified some of their results in \cite{seibertcomplex}.

An additional viewpoint was opened with the invention of tight closure in the late eighties.
Without using explicitly the terminology of Hilbert-Kunz multiplicities, M. Hochster and C. Huneke showed in
\cite[Theorem 8.17]{hochsterhuneketightclosure}
(see also \cite[Theorem 5.4]{hunekeapplication} for the formulation in Hilbert-Kunz terminology)
under weak conditions on the local ring
for finitely generated submodules $N \subseteq W \subseteq M$ with $W/N$ artinian  that
$W \subseteq N^*$,
the tight closure of $N$ (inside $M$), holds if and only if the limit of the lengths of the quotient 
$ \lim_{e \rightarrow \infty}   \frac{ \lg \left(  W^{[p^e]} /N^{[p^e]} \right)}{p^{e \dim R}} = 0$.
This viewpoint of \emph{(minimal) relative Hilbert-Kunz multiplicity} was studied in \cite{watanabeyoshidaminimal}
and a relation to the $F$-\emph{signature} was established.
The desire to have a Hilbert-Kunz criterion for tight closure even when the ideal (or the module) is not (co-)primary and related
questions led to the study of zeroth local cohomology of Frobenius pull-backs of modules by I. Aberbach in \cite{aberbachgorenstein}
and
N. Epstein and Y. Yao in \cite{epsteinyao}. Finally, H. Dao and I. Smirnov considered in \cite{daosmirnov} local-cohomological
Hilbert-Kunz multiplicities in general and proved that they exist under some weak conditions;
in particular they exist if the ring has an isolated singularity (see \cite[Corollary 3.7]{daosmirnov}).
\end{rmk}

\section{Hilbert-Kunz function in the graded case}
\label{geometric}

In this section we describe how the Hilbert-Kunz function of a graded module of finite length can be described in the graded case with the help of
vector bundles over the corresponding projective variety.
This approach was very successful in ring dimension two, when the corresponding projective varieties are curves, see
\cite{brennerrational}, \cite {trivedirational}.
In higher dimension one should not expect results which settle everything; instead one needs a detailed
study of specific projective varieties in order to make progress.

The  degree of a polarized variety $Y$ of dimension $d$ with fixed ample divisor $H$
is the self intersection number $H^d$.
For a hypersurface $Y \subset {\mathbb P}^{d+1}$ endowed with the hyperplane section this is
just the degree of the defining equation.
For a standard-graded ring $R$ we always use the polarization of $Y=\Proj R$ given by $\O_Y(1)$.

The following theorem gives a general translation from Hilbert-Kunz function to data on the projective variety.
It is only useful if we can find ways to control the top cohomology of the top syzygy bundle. As usual, we set
$h^{i} (\shF)= \dim_K H^i(Y, \shF)$ for a coherent sheaf $\shF$ on $Y$.

\begin{thm}
\label{hilbertkunzgeometric}
Let $R$ be a standard-graded Cohen-Macaulay domain of dimension $d +1 \geq 2$ with an isolated singularity
over an algebraically closed field $K$ of positive characteristic $p$. Let $H^d$ denote the degree of $Y=\Proj R$.
Let $M$ be a graded $R$-module of finite length. Let
\[ \cdots  \longrightarrow F_2  \longrightarrow F_1 
\longrightarrow  F_0  \longrightarrow M \longrightarrow 0 \]
be a graded complex which is exact (as sheaves) on $U=D(R_+)$, where $F_i =\bigoplus_{j \in J_i}   R(- \beta_{i j})$
are graded free $R$-modules
(we call such a complex a \emph{punctured resolution} of $M$).
Let $\operatorname{Syz}_i = \ker \delta_i $, where $\delta_i:F_i \rightarrow F_{i-1}$.
We denote the corresponding modules on $U$ and on $Y$ with the same symbols.
Then
\begin{eqnarray*}
 & & \operatorname{HKF}(M,e) \\
 & \!\!  \!\!  = \!\!  \!\!  & \sum_{m \in {\mathbb N} }  h^d (( F^{e*} \Syz_d) (m)  )
+ \sum_{i=0}^{d} (-1)^{ d-1-i } \left(    \sum_{m \in \N    }    h^d( ( F^{e*} F_i )  ( m))     \right) 
\cr
& \!\! \!\!   = \!\!  \!\!  &  \!\!  \sum_{m \in {\mathbb N} }  h^d ( (F^{e*} \Syz_d) (m)  )
  \!  +  \!    \sum_{i=0}^{d} (-1)^{ d-1-i }   \!\!     \left(    \sum_{j \in J_i } \!
\left( \sum_{m \in \N    }   h^d( \O_Y( - \beta_{ij} q + m))   \!\!   \right)   \!\!        \right) 
\end{eqnarray*}
(everything is computed on $Y$)
and
\begin{eqnarray*}
& & \operatorname{HK}(N,M) \\
& =& \!\!  \lim_{e \rightarrow \infty} \frac{ \sum_{m \in {\mathbb N} }  h^d (( F^{e*} \Syz_d) (m)  )}{q^{d+1} }
+ \frac{H^d}{(d+1)!}  \!\!   \left(  \sum_{i=0}^{d} (-1)^{ d+1-i }  \left(       \sum_{j \in J_i } 
 \beta^{d+1}_{ij}  \!\!    \right)  \!\!  \right)  \!\!  .
\end{eqnarray*}
\end{thm}
\begin{proof}
The module $ F^{e*}M$, whose length (or dimension over $K$) we would like to compute, is the cokernel of the morphism
\[ F^{e*}  F_1 \stackrel{\delta^q_1}{ \longrightarrow}  F^{e*}    F_0 \, .\]
Since the dimension of $R$ is at least three and $R$ is supposed to be normal, this homomorphism equals
\[  \Gamma(U,  F^{e*}    F_1)  \stackrel{\delta^q_1}{ \longrightarrow}  \Gamma(U,  F^{e*}           F_0) \, . \]
On $U$ we have the short exact sequence
$0 \rightarrow \Syz_1 \rightarrow F_1 \rightarrow F_0 \rightarrow 0$, since $M$ is supported on $\m$. The
Frobenius pull-backs of this sequence are exact and so the long exact sequence of sheaf cohomology on $U$ for
$0 \rightarrow  F^{e*}\Syz_1 \rightarrow F^{e*} F_1 \rightarrow  F^{e*} F_0 \rightarrow 0$
yields for this cokernel
\[  0 \longrightarrow \koker_e    \longrightarrow H^1 (U,  F^{e*} \Syz_1   )
\longrightarrow H^1 (U,  F^{e*}  F_1   )  \longrightarrow  \cdots  \, .\]
If the dimension of $R$ is at least three,
then by Cohen-Macaulayness we have $H^1 (U,  F^{e*}  F_1   ) =0$ (the proof for the two-dimensional case continues below)
and therefore we have $  F^{e*}M \cong H^1(U, F^{e*} \Syz_1 ) $.

On the projective variety $Y$ we have also the short exact sequences
\[   0 \longrightarrow \Syz_{1} \longrightarrow F_1 \longrightarrow F_0 \longrightarrow 0 \, \]  
and
\[   0 \longrightarrow \Syz_{i+1} \longrightarrow F_{i+1} \longrightarrow \Syz_{i} \longrightarrow 0 \, \]
for $i \geq 1$
(for convenience we set $\Syz_0=F_0$)
where now $F_i \!= \! \bigoplus_{j \in J_i} \! \O_Y(-\beta_{ij})$.
Inductively we see that the $\Syz_i$ are locally free.
Applying the Frobenius pull-back to these sequences and twisting by $\O_Y(m)$ yields
\begin{eqnarray*}
  \koker_e & = &  \bigoplus_{m \in \N} 
\koker \left(     \Gamma(Y, ( F^{e*} F_1)(m)) \longrightarrow \Gamma (Y, (F^{e*} F_0) (m) )    \right) \\
&=&   \bigoplus_{m \in \N}      \ker \left(  H^1(Y, ( F^{e*} \Syz_1)(m))  \longrightarrow    H^1(Y, ( F^{e*} F_1)(m))           \right) \\
&=&   \bigoplus_{m \in \N}       H^1(Y, ( F^{e*} \Syz_1)(m))  \, . 
\end{eqnarray*}
We also get exact sequences
\[ \longrightarrow  H^i(Y, (F^{e*} F_{i+1}) (m) )      \longrightarrow
H^{i}(Y,(F^{e*} \Syz_i) (m)) \]
\[ \longrightarrow
 H^{i+1} (Y, (F^{e*} \Syz_{i+1} ) (m))  \longrightarrow   H^{i+1}(Y, (F^{e*} F_{i+1}) (m) )    \longrightarrow       \]
which induce isomorphisms
\[  H^{i}(Y,(F^{e*} \Syz_i) (m)) \cong H^{i+1} (Y, (F^{e*} \Syz_{i+1} ) (m))\]
for $i=1, \ldots , d-2$ (this is empty for $d=1,2$). Hence the Hilbert-Kunz function is the sum over all $m$
of the dimensions of either of these cohomology modules. Moreover, for $i=d-1$ we get the exact sequence
\[ 0   \longrightarrow  H^{d-1}(Y, (F^{e*} \Syz_{d-1})(m))     \longrightarrow  H^{d}(Y, (F^{e*} \Syz_d)(m))  \]
\[\longrightarrow   H^d(Y,  ( F^{e*} F_{d}) (m) )  \longrightarrow    H^{d}(Y, (F^{e*} \Syz_{d-1})(m))   \longrightarrow 0  \, . \]
The Hilbert-Kunz function is the sum of the left hand module over all $m$,
hence it can be expressed by the sum over $m$ of the alternating sum of
the other expressions
(for $d=1$ there is no $0$ on the left, but also in this case the Hilbert-Kunz function is given by this alternating sum).
The terms $\sum_{m \in \N} H^{d}(Y, (F^{e*} \Syz_d)(m)) $ and $ H^d(Y, (F^{e*} F_{d}) (m) ) $
(with a minus sign)
are explicitly stated in the formula of the theorem
(for $d=1$ this is also true for the last module, finishing the proof in this case).

The term  $  H^{d}(Y, (F^{e*} \Syz_{d-1})(m))   $ will be computed using again the defining short exact sequences.
Note first that $H^j(Y, (F^{e*} \Syz_i)(m))=0$ for $i=1, \ldots, d-1$ and all $j$ with $ i+1  \leq  j \leq d-1 $.
We prove this claim by induction on $i$. For $i=1$ this follows from 
\[H^{j-1}(Y, (F^{e*} F_0)(m) ) \longrightarrow H^j(Y,  (F^{e*}  \Syz_1)(m) ) \longrightarrow   H^j(Y,(F^{e*}  F_1)(m) ) \]
(coming from the first defining sequence)
and the Cohen-Macaulay property.
The induction step follows from
\[H^{j-1}(Y, (F^{e*} \Syz_{i}  )(m) ) \rightarrow H^j(Y,  (F^{e*}  \Syz_{i+1})(m) ) \rightarrow   H^j(Y,(F^{e*}  F_{i+1})(m) ) .\]
From this claim we deduce the short exact sequences
\[0 \rightarrow \! H^{d} (Y, (\!F^{e*} \Syz_{i}\!  )(m)\! ) \rightarrow \! H^d(Y,  (\!F^{e*}  F_i\! )(m)\! )
\rightarrow \!   H^d(Y,(\! F^{e*} \Syz_{i-1}\! )(m)\! ) \rightarrow \! 0 \]
for $i \leq d-1$
and we compute
\begin{eqnarray*}
\!\!\!\!\! & &\!\!\! \!\! \sum_{m=0}^\infty h^{d} ( (F^{e*} \Syz_{d-1} )(m)   )   \\
\!\!\!\! \! & = & \!\!\!\!\! \sum_{m=0}^\infty h^{d} ( (F^{e*} F_{d-1} )(m) )   -  \sum_{m=0}^\infty h^{d} ( (F^{e*} \Syz_{d-2} )(m) \\
\!\!\!\!  \!&=& \!\!\! \!\!  \sum_{m=0}^\infty h^{d} ( (F^{e*} F_{d-1} )(m) ) \!
- \! \!\left( \sum_{m=0}^\infty  \!\!   h^{d} ( (F^{e*} F_{d-2} )(m) )\!
- \!\! \! \sum_{m=0}^\infty   \!\!        h^{d} ( (F^{e*} \Syz_{d-3} )(m)  \!\!  \right)                \\
\!\!\!\!  \! &=& \vdots \\
\!\!\!\! \! &=& \!\!\!\!\! \sum_{i=0}^{d-1} (-1)^{d+1-i}    \left(   \sum_{m=0}^\infty h^{d} ( (F^{e*} F_{i} )(m) )       \right) \, .
\end{eqnarray*}
This gives the first equation in the formula for the Hilbert-Kunz function. The second equation follows immediately using
$F_i = \bigoplus_{ j \in J_i} \O_Y(- \beta_{ij})$.

For the formula for the Hilbert-Kunz multiplicity we only have to compute
\[\lim_{e \rightarrow \infty}  \frac{1}{q^{d+1} }   \sum_{m=0}^\infty    h^d( \O_Y(- \beta q +m) )  \, ,     \]
and so Serre duality (\cite[Corollary III.7.7]{hartshorne}) and Riemann-Roch (\cite[Corollary 15.2.1]{fultonintersection})
gives that this is $\frac{H^d}{(d+1)!}   \beta^{d+1} $
(see Lemma \ref{antiamplethreshold} below for this argument in a slightly more complicated setting).
\end{proof}

Because we will focus in examples on the case of ring dimension three, we state the following corollary explicitly.

\begin{cor}
\label{hilbertkunzgeometricsurface}
Let $R$ be a three-dimensional standard-graded Cohen-Ma\-cau\-lay domain with an isolated singularity 
over an algebraically closed field of positive characteristic $p$. Let $H^2$ denote the degree of $Y=\Proj R$.
Let
$I=(f_1, \ldots, f_n) \subseteq R$ be a homogeneous $R_+$-primary ideal with $d_i= \deg (f_i)$. Let
\[ 0\longrightarrow \Syz_2 \longrightarrow  F_2\!  = \! \bigoplus_{j=1}^s \! R(- \beta_j)  \longrightarrow 
F_1 \!= \! \bigoplus_{i=1}^n R(-d_i) 
\longrightarrow  R  \longrightarrow R/I \longrightarrow 0 \]
be a graded complex which is exact on $D(R_+)$.
Then
\[ \operatorname{HK}(I)
= \lim_{e \rightarrow \infty} \frac{ \sum_{m \in {\mathbb N} }  h^2 (( F^{e*} \Syz_2) (m)  )}{q^{3} }
+ \frac{H^2}{6}\left(  -  \sum_{j=1}^{s}  \beta_j^3  +   \sum_{i=1}^n  d_i^3
   \right)       \, . \]
\end{cor}
\begin{proof}
This follows directly from Theorem \ref{hilbertkunzgeometric}. 
\end{proof}

\begin{rmk}
\label{serredual}
The dimension of the top cohomology \[H^d (Y, (F^{e*} \Syz_d )(m))\] can be computed by Serre duality as the dimension of the sections
\[H^0(Y, (F^{e*} ( \Syz_d^\dual )) (-m) \otimes \omega_Y  ) \, .\]
In summing up over all nonnegative $m$ and dividing by $q^{d+1}$,
the effect of the canonical sheaf will vanish, so we are interested in the limit
\[\lim_{e \rightarrow \infty} \frac{ \sum_{m=0}^{- \infty} h^0((F^{e*} \shG) (m) ) }{  q^{d+1} }\]
where $\shG$ is locally free.
This is a kind of Frobenius-Riemann-Roch problem and still difficult in general. One can only expect to solve this problem for
specific classes of varieties and bundles. We will see later (Corollary \ref{oguisoirrational})
that this limit can be an irrational number even for line bundles on
smooth projective hypersurfaces of degree four in ${\mathbb P}^3$.
\end{rmk}

\begin{rmk}
The Cohen-Macaulay assumption in Theorem \ref{hilbertkunzgeometric} is not essential.
We know that the intermediate cohomology of
$\O_Y(m)$, that is $H^i(Y, \O_Y(m))$ for $i=1, \ldots , \dim Y-1$, lives only in a finite range for $m$ (as long as $Y$ is smooth).
Hence the sequences with which we work in the proof 
of the theorem may not be exact anymore, however, their unexactness is neglectable by dividing through $q^{\dim Y +1}$.
Also on singular normal varieties
the theorem holds, provided that Serre duality and an appropriate version of Riemann-Roch holds.
\end{rmk}

\begin{rmk}
\label{finitepd}
If $M$ has finite projective dimension over $R$,
then we can take the minimal free resolution and we get
$\Syz_d \cong \bigoplus_{j \in J_{d+1} } \O_Y(-\beta_{d+1, j})  $ ($=F_{d+1}$). In this case the formula from Theorem
\ref{hilbertkunzgeometric} for the Hilbert-Kunz multiplicity becomes
\[ \operatorname{HK}(M)
= \frac{H^d}{(d+1)!}\left(  \sum_{i=0}^{d+1} (-1)^{ d+1-i }  \left(       \sum_{j \in J_i } 
\beta^{d+1}_{ij}  \right)   \right) \, . \]
This also holds if we have a finite punctured resolution.
\end{rmk}

\begin{rmk}
If $M$ has infinite projective dimension over $R$, and we take the minimal resolution over $R$
then we know that $\Syz_d  $ has no intermediate cohomology, i.e. no cohomology apart from the $0$th and the top-dimensional cohomology.
So the corresponding module is a maximal Cohen-Macaulay module.
In some cases (\cite{casanellashartshorne}, \cite{faenzicubic}, \cite{malaspinarao}, \cite{mohankumarraoravindra}, \cite{ottavianispinor})
these are reasonably well understood. Note however that the Frobenius pull-back of such a bundle will have intermediate
cohomology in general.
\end{rmk}

\begin{rmk}
Let $R=K[X_0, \ldots , X_{N}]/{\mathfrak a}$ be a graded ring as in the theorem. Then $M$ has finite projective dimension
over the polynomial ring $K[X_0, \ldots, X_{N}]$. Let $F_\bullet$ be the
finite minimal resolution and set $\Syz_i = \ker \delta_i $ (on ${\mathbb P}^N$).
Note that $\Syz_1 = \Omega_{{\mathbb P}^N}$. The restriction $F_\bullet |_Y $, $Y=V({\mathfrak a})$,
is a complex fulfilling the assumptions of the theorem, even if we loose minimality and global exactness.
But it seems difficult to use this for computations. If $R$ is a hypersuface ring,
then $\Syz_d= \koker (F_{d+1} \hookrightarrow F_{d} )$, and if $M=R/I$ is Gorenstein,
then the minimal resolution over the polynomial ring
is symmetric \cite[Corollary 21.16]{eisenbud} and then
this module is dual to $\Syz_1$. 
\end{rmk}

\begin{rmk}
If $I=(f_1, \ldots ,f_n)$ is a homogeneous $R_+$-primary ideal,
then one can always take the Koszul resolution on these elements to get a complex 
as required by the theorem. 
\end{rmk}

\begin{rmk}
For $d=1$, when $R$ is a two-dimensional standard-graded domain and $Y= \Proj R$ the corresponding smooth curve,
the asymptotic behavior of the top-dimensional term
$\frac{\sum_{m \in \N} h^1( (F^{e*} \Syz_1)(m)  )}{q^2}$ is completely encoded in the strong Harder-Narasimhan filtration
of $\Syz_1$, which exists by \cite[Theorem 2.7]{langer}.
With this observation it was shown in \cite{brennerrational} and \cite{trivedirational} that the Hilbert-Kunz
multiplicity is a rational number.
This method can not be directly adopted to higher dimension,
the difficulty is that with stability one can control the vanishing of global sections and of
the top-dimensional cohomology, but not the intermediate cohomology.
Despite of this, the stability of the syzygy bundles is an important property
in order to understand the Hilbert-Kunz function.
For $\Syz_1 = \left( \Omega_{{\mathbb P}^N} \right)|_{\Proj R}$ for a suitable embedding $\Proj R \subseteq {\mathbb P}^N$,
there are results saying that the stability of the cotangent bundle is preserved when restricting to $\Proj R$
(see \cite{flenner}, \cite{mehtaramanathan} for restriction to curves and \cite{cameresurface}, \cite{einlazarsfeldmustopa}
for results on restricting to surfaces).
\end{rmk}

With some effort, the following example could be extended to all dimensions, but we stick to graded rings of Krull dimension three.

\begin{example}
Let $f_1,f_2,f_3$ be homogeneous parameters of degree $d_1$, $d_2$, $d_3$
in a standard-graded Cohen-Macaulay domain $R$ of dimension $3$ with an isolated singularity
of positive characteristic $p$ and let
$S = \Proj R$ be the corresponding projective surface. Another homogeneous element $f$ of degree $\ell $ yields 
the ideal $I=(f_1,f_2,f_3,f)$. If $\ell \geq d_1+d_2+d_3$, then $f$ belongs to the tight closure of the parameter ideal
$(f_1,f_2,f_3)$ by \cite[Theorem 2.9]{hunekeparameter}
and hence by \cite[Theorem 5.4]{hunekeapplication} these two ideals have the same Hilbert-Kunz multiplicity.
We want to show how this can be seen in our geometric approach.
The Hilbert-Kunz multiplicity of the parameter ideal is
\begin{eqnarray*}  
 \!\!\!\! & &\!\!\!\! \!\!\!\! \!\! \!  \frac{H^2}{6} ( d_1^3+d_2^3+d_3^3  - ((d_1+d_2)^3  +(d_1+d_3)^3    +(d_2+d_3)^3  ) + (d_1+d_2+d_3)^3 )  \\
 \!\!\!\! &=&\frac{H^2} {6} (6 d_1d_2d_3     ) \\
 \!\!\!\! &=&   H^2 d_1d_2d_3     
\end{eqnarray*}
by Remark \ref{finitepd}.
For the computation of the Hilbert-Kunz multiplicity of $I$
we look at the following commutative diagram of locally free sheaves on $Y$,
with exact rows and columns (we set $\Syz_1= \Syz (f_1,f_2,f_3,f)$),
\[
\begin{matrix}
   & 0 &  & 0& &0 &   \\
  & \downarrow&  & \downarrow & &\downarrow &   \\
0  \rightarrow  & \O(-d_1-d_2 -d_3)  & \rightarrow  & \Syz_2 & \rightarrow & \Syz(f_1,f_2,f_3)(-\ell) & \rightarrow   0  \\
   & \downarrow&  & \downarrow & &\downarrow &   \\
0  \rightarrow  & \!\!\!\!  \! \O ( \!- \!d_1 \!- \!d_2 \!) \! \oplus  \!  \O( \!- \!d_1 \!- \!d_3 \!)  \!
\oplus  \! \O( \!- \!d_2 \!- \!d_3 \!) \!\!\!\!  \!
& \rightarrow  &  F_2
& \rightarrow &  \!\!\!\! \!    \O(-\ell -d_1) \oplus   \O(-\ell-d_2) \oplus  \O(-\ell-d_3)  \!\!\!\!  \!    & \rightarrow  0  \\
   & \downarrow&  & \downarrow & &\downarrow &   \\
0   \rightarrow  & \Syz(f_1,f_2,f_3) & \rightarrow   & \!\!\!\!  \! \Syz_1  \!\!\!\!  \! & \rightarrow & \O (-\ell) & \rightarrow   0  \\
   & \downarrow&  & \downarrow & &\downarrow &   \\
   & 0 &  & 0& &0\, . \!\!\! & &  
\end{matrix}
\]
In the columns we have Koszul resolutions, and $F_2  $ is the direct sum of the left and of the right.
From the top row we get for each $q=p^e$, $m \geq 0$, a homomorphism
\[ H^1(S, (F^{e*}\Syz(f_1,f_2,f_3)) (-q \ell +m))  \longrightarrow   H^2(S,   \O_S (m -qd_1-qd_2 -qd_3) ) \, . \]
For $m < q \ell$ the left hand side is $0$, because the first cohomology stems from $H^0(S, \O_S (- \ell q +m))$ which
lives only in nonnegative degrees.
So if we suppose $\ell \geq d_1+d_2+d_3$, then we have  $m \geq  \ell q \geq (d_1+d_2+d_3) q $.
It follows that all nonzero cohomology of
$ H^1(S, (F^{e*}\Syz(f_1,f_2,f_3)) (- \ell q +m))  $ is sent to the nonnegative degree range of
$ H^2(S,   \O_S(m -qd_1-qd_2 -qd_3) ) $. This degree range is finite independently of $q$. Hence the cokernel
of this homomorphism contains the complete range of $ H^2(S,   \O_S(m - d_1 q-d_2 q-d_3q) ) $ from $m=0$ up to $m < (d_1+d_2+d_3)q$.
Hence the kernel of the surjection
\[ H^2(S, (F^{e*} \Syz_2)(m) )     \longrightarrow   H^2(S, (F^{e*} \Syz_1(f_1,f_2,f_3) ) (- \ell q + m) )  \]
contains this range.
Therefore we have asymptotically
\[\frac{ \sum_{m \in {\mathbb N} }  h^2 (( F^{e*} \Syz_2) (m)  )}{q^{3} } \]
\[ \sim  \frac{ \sum_{m \in {\mathbb N} }  h^2 (( F^{e*} \Syz(f_1,f_2,f_3)) (-\ell q + m) )}{q^{3} } + \frac{H^2}{6}(d_1+d_2+d_3)^3
\, . \]
From the right hand column
\[0 \longrightarrow \Syz(f_1,f_2,f_3) (-\ell)
\longrightarrow \bigoplus_{i=1}^3  \O_S(-\ell-d_i)  \longrightarrow \O(-\ell) \longrightarrow 0 \] 
we deduce
\begin{eqnarray*}
& & \sum_{m=0}^\infty   h^2( (F^{e*} \Syz(f_1,f_2,f_3) )     (-\ell q +m)  ) \\
&=& \sum_{i = 1}^3 \left(               \sum_{m=0}^\infty   h^2(\O_S    ( -  (d_i+\ell)q  +m )  )           \right)
-  \sum_{m=0}^\infty   h^2( \O_S (-  \ell q +m )  )  \\
&=& \frac{H^2}{6} q^3 \left(  (d_1 +\ell)^3  +  (d_2 +\ell)^3 +  (d_3 +\ell)^3      - \ell^3    \right)  +O(q^2 )    \\
&=& \!\! \! \! \frac{H^2}{6} q^3 \! \left(\!  d_1^3 +d_2^3 +d_3^3  +   3    \ell (d_1^2+d_2^2 +d_3^2) \! + \!
3 \ell^2 (d_1+d_2+d_3)  + \! 2 \ell^3 \! \right) \! + \! O(q^2 ) .
\end{eqnarray*}
Finally, Theorem \ref{hilbertkunzgeometric} gives the Hilbert-Kunz multiplicity of $I$ as
\begin{eqnarray*}
  \!\!\!\!\!\! & \!\!\! & \lim_{e \rightarrow \infty} \frac{\sum_{m \in \N}   h^2( (F^{e*} \Syz_2)(m)) }{q^3} \\
 \!\!\! \!\!\! & \!\!\! & +  \frac{H^2}{6} \left(  d_1^3 + d_2^3 + d_3^3 +\ell^3   -   \sum_{i \neq j} (d_i+d_j)^3 
 - \sum_{i =1}^3  (  d_i+  \ell)^3       \right)  \\
 \!\!\! \!\!\! &=&  \!\!\!  \lim_{e \rightarrow \infty} \frac{\sum_{m \in \N}   h^2( (F^{e*} \Syz(f_1,f_2,f_3)(-q \ell + m)) }{q^3}     \\
 \!\!\! \!\!\! & & \!\!\! \!\!\! + \frac{H^2}{6}  \!\! 
 \left(   \!\!   (d_1+d_2+d_3)^3 +   d_1^3 + d_2^3 + d_3^3 +\ell^3   -  \!\!  \sum_{i \neq j} (d_i+d_j)^3 
  \!\! - \sum_{i =1}^3  (  d_i+  \ell)^3   \!\!  \right) \\
 \!\!\! \!\!\! &=&  \frac{H^2}{6} \left(  d_1^3 +d_2^3 +d_3^3   
 +3 \ell (d_1^2+d_2^2 +d_3^2)  +3 \ell^2 (d_1+d_2+d_3)  + 2 \ell^3   \right) \\
 \!\!\! \!\!\! & & \!\!\! \!\!\! + \frac{H^2}{6}  \!\! 
 \left(   \!\!   (d_1+d_2+d_3)^3 +   d_1^3 + d_2^3 + d_3^3 +\ell^3   - \!\!  \sum_{i \neq j} (d_i+d_j)^3 
  \!\! - \sum_{i =1}^3  (  d_i+  \ell)^3   \!\!  \right) \\
  \!\!\!   & = & H^2  d_1 d_2 d_3   \, ,
\end{eqnarray*}
so the two Hilbert-Kunz multiplicities coincide.
\end{example}

The theorem justifies the following definition.

\begin{defn}
Let $K$ be an algebraically closed field of positive characteristic $p$ and let
$Y$ be a polarized projective variety of dimension $d$ with fixed very ample invertible sheaf $\O_Y(1)$.
Let $\shG$ be a coherent sheaf on $Y$. Then for $i \geq 1$ we set
\[ \operatorname{HK}^i (\shG) =\lim_{e \rightarrow \infty}  \frac{ \sum_{m=0}^\infty h^i((F^{e*} \shG)(m) ) }{  p^{e (d+1)} }\]
and call it the $i$th \emph{(sheaf-)cohomological Hilbert-Kunz multiplicity} of $\shG$.
\end{defn}

\begin{rmk}
The numerators in the definition are finite for each $e$, since the $h^i ((F^{e*} \shG)(m) )$ are $0$
outside a finite range for $m$. We do not give here a systematic treatment of these numbers.
For $i$ between $1$ and $d-1$ one may also consider the sum over all $m\in \Z$ (at least for $\shG$ locally free and $Y$ smooth).
For $i=0$, the expression
$\lim_{e \rightarrow \infty}   \frac{ \sum_{m=0}^{- \infty} h^0((F^{e*} \shG)(m) ) }{  p^{e (d+1)} }$ is the right one to look at
(see also Remark \ref{serredual}).
The proof of Theorem \ref{hilbertkunzgeometric} shows that for a normal standard-graded domain of dimension $d+1 \geq 3$ and a graded
$R$-module $M$ of finite length its Hilbert-Kunz multiplicity equals the $i$th  cohomological Hilbert-Kunz multiplicity of the $i$th syzygy bundle
for $i=1, \ldots, d-1$. The expression $\lim_{e \rightarrow \infty}  \frac{h^i( F^{e*}(\shG)) }{p^{e d}}$  might be called the $i$th
\emph{cohomological Frobenius-volume}. 
\end{rmk}

\section{Resolutions on two-dimensional hypersurfaces}
\label{resolution}

The following lemma shows that there are interesting cases where the second syzygy bundle for a specific punctured resolution
has small rank.

\begin{lemma}
\label{hypersurfaceresolution}
Let $K$ denote an algebraically closed field and let $ R = $ \linebreak $ K[X,Y,Z,W]/(F)$ with $F$ a homogeneous polynomial of degree $\delta$ and such
that $S=\Proj R \subset {\mathbb P}^3_K$ is smooth. Suppose furthermore that the lines
$V_+(X,W)$, $V_+(X,Z)$, $V_+(Y,W)$ and $V_+(Y,Z)$ meet the surface in exactly $\delta$ points.
Then the first syzygy bundle $\Syz(X,Y,Z,W)$ (corresponding to the maximal ideal) sits inside
the short exact sequence (on $S$)
\[ 0 \longrightarrow \shE \longrightarrow   \O_S(-\delta)  \oplus  \bigoplus_4 \O_S(-2)  \longrightarrow \Syz(X,Y,Z,W) \longrightarrow 0 \, .\]
Here $\shE$ has rank two, its determinant is $\det \shE = \O_S(-4-\delta)$,
its degree is $ (- 4 - \delta )  \delta $ and its second Chern class is  $ (2+4 \delta ) H^2 $.
Moreover, the sequence
\[ 0 \longrightarrow \shE \longrightarrow   \bigoplus_4 \O_S(-2)    \longrightarrow \Omega_S  \longrightarrow 0 \]
is exact.
\end{lemma}
\begin{proof}
We have the short exact sequence (see \cite[Theorem II.8]{hartshorne})
\[0 \longrightarrow \O_S(-\delta) \longrightarrow
\Syz(X,Y,Z,W) \cong \left(\Omega_{ {\mathbb P}^3} \right)|_S \longrightarrow \Omega_S \longrightarrow 0 \]
and the (sheaf) surjection
\[ \bigoplus_6 \O_S(-2) \longrightarrow  \Syz(X,Y,Z,W) \longrightarrow 0 \,   \]
coming from the Koszul resolution. We claim that the four Koszul syzygies
\[ (Y,-X,0,0),\,( Z,0,-X,0 ),\,(0,W,0,-Y), \, (  0,0, W, -Z) \]
together with the differential syzygy
$\left(\frac{ \partial F}{ \partial X },\, \frac{ \partial F}{ \partial Y },\,
\frac{ \partial F}{ \partial Z },\,\frac{ \partial F}{ \partial W }\right)$
coming from the surface equation define already a surjection
\[ \O_S(-\delta) \oplus \bigoplus_4 \O_S(-2) \longrightarrow \Syz(X,Y,Z,W)  \longrightarrow  0 \, .\]
Since the differential syzygy vanishes on the surface,
this implies also that these four Koszul syzygies surject onto $\Omega_S$ and that the two kernels are the same.
To prove the claim, we show for every point $P\in S$ that all five syzygies together
span a three-dimensional subspace of $\Syz(X,Y,Z,W)$ ($\subset \bigoplus_4 \O_S(-1)$). 
For a point $P$ where $X, Y \neq 0$ or $W,Z \neq 0$ this is clear.
So assume that say $X=W=0$. Then at least one of $Y$, $Z$ is not zero, and the four Koszul syzygies evaluated at $P$ are
$(Y(P), 0,0,0)$, $( Z(P),0,0,0 )$, $(0,0,0,-Y(P))$ and $( 0,0, 0, -Z(P))$, so they only give a two-dimensional subspace
of $\Syz(X,Y,Z,W)_P$. The second and the third  component of the differential syzygy are 
$\frac{ \partial F}{ \partial Y } (P)$ and $\frac{ \partial F}{ \partial Z } (P) $.
We write the surface equation as
$F=XG+WH+Q(Y,Z)$ and we have $Q(Y,Z)\neq 0$   
(else $F \in (X,W)$ and the line $V_+(X,W)$ would lie on the surface). Therefore
$ \frac{ \partial F}{ \partial Y } =  X \frac{ \partial G}{ \partial Y } 
+ W \frac{ \partial H}{ \partial Y } +  \frac{ \partial Q(Y,Z)}{ \partial Y } $.
Plugging in $P$
gives $   \frac{ \partial F}{ \partial Y } (P)=   \frac{ \partial Q(Y,Z)}{ \partial Y } (P)  $ and similarly
$  \frac{ \partial F}{ \partial Z } (P)=   \frac{ \partial Q(Y,Z)}{ \partial Z } (P)  $. 
Since $Q$ has exactly $\delta$ zeros, these partial differentials can not both vanish at $P$, hence the differential syzygy
contributes with a new dimension.

The statement about the determinant and the degree of $\shE$ follows from the just proven short exact sequence.
The Chern polynomial (see \cite[Appendix A]{hartshorne})
of $\Syz(X,Y,Z,W)$ is $(1-Ht)^4 = 1 -4 Ht + 6 H^2 t^2$ (in the Chow ring). Hence  by
looking at
\[   (1- \delta Ht)( 1+ (  \delta-4  ) H t +c_2(\Omega_S) t^2  )      =    1 -4 Ht + 6 H^2 t^2          \]
we deduce $c_2(\Omega_S)=    ( \delta^2  -4 \delta  + 6  ) H^2 $.
To compute the second Chern class of $\shE$ we look at
\[ \!  \left(     \!            1 \!  + \!   (\!   -4 \!   -  \!  \delta  )  H t   \!   +  \!   c_2(\shE)    t^2   \!         \right) \! \! 
\left(     1+ (  \delta \!   - \!   4  ) H t +     ( \delta^2 \!   - \!   4 \delta \!  +\!  6  )\!  H^2     t^2    \right)
\!  = \!  1         -       8Ht  +        24H^2 t^2\!  .  \]    
This gives
\[c_2(\shE) =  (    24 - (\delta -4)( - \delta -4)   - \delta^2 + 4 \delta - 6 ) H^2  
= (2 + 4 \delta ) H^2 \, .  \]
\end{proof}

Even on projective spaces the concepts of minimal resolution and punctured resolution differ, as the following easy example shows.

\begin{example}
Consider the ideal $I=(X^2,Y^2,Z^2,XY)$ in $K[X,Y,Z]$. The syzygy bundle $\Syz (X^2,Y^2,Z^2,XY)$ on ${\mathbb P}^2$
has the globally surjective
resolution
\[    \O_{  {\mathbb P}^2      } (-3)^{\oplus 2}  \oplus    \O_{  {\mathbb P}^2      } (-4)^{\oplus 4}         \longrightarrow   \Syz (X^2,Y^2,Z^2,XY)  \]
given by the monomial syzygies
\[ (Y,0,0,-X),\,  (0,X,0,-Y),\,   (0,0,XY,-Z^2),\, \]
\[  (Y^2,-X^2,0,0), \,    (Z^2,0,-X^2,0),\,  (0,Z^2,-Y^2,0) \,        .    \]
We claim that after removing $  (Y^2,-X^2,0,0) $
we still have a sheaf surjection
\[    \O_{  {\mathbb P}^2 } (-3)^{\oplus 2}  \oplus    \O_{  {\mathbb P}^2      } (-4)^{\oplus 3}    \longrightarrow   \Syz (X^2,Y^2,Z^2,XY)  \]
and hence a punctured resolution. For this we have to check that the remaining five syzygies span at every point
$P \in {\mathbb P}^2$ a three-dimensional subspace of $\Syz (X^2,Y^2,Z^2,XY)_P$.
This is clear for $X,Y \neq 0$, for $X=0$, $Z \neq 0$
and for $X =Z= 0$, $Y \neq 0$.
\end{example}

\section{Ample and antiample threshold}
\label{amplethreshold}

In Section \ref{geometric} we have seen that the Hilbert-Kunz multiplicity of a graded module of finite length
over a standard-graded ring $R$
has led us to the asymptotic 
consideration of
$\sum_{m=0}^\infty \dim_K H^d(Y, (F^{e*}\shG)(m))$ in dependence of $e$ for certain locally free sheaves $\shG$ on $Y=\Proj R$.
In this section we deal with the easiest case, when $\shG= \shL$ is an invertible sheaf. 
Similar expressions like the ones in the following definitions occur in \cite[Section 2.3.B]{lazarsfeldpositivity1}.

\begin{defn}
Let $Y$ be a smooth projective polarized variety with fixed very ample divisor $H$.
For a divisor $L$ we define
\[ \at (L) :=  \operatorname{inf} \left \{ \frac{m}{n} :\, mH + nL   \text{ is ample} \right \}\]
and call it the \emph{ample threshold} of $L$.
\end{defn}

\begin{defn}
Let $Y$ be a smooth projective polarized variety with fixed very ample divisor $H$.
For a  divisor $L$ we define
\[ \aat(L) := \operatorname{sup} \left \{ \frac{m}{n} :\, mH+ nL \text{ is antiample} \right \}\]
and call it the \emph{antiample threshold} of $L$.
\end{defn}

\begin{rmk}
Note that ampleness is not affected by taking a positive multiple (see \cite[Proposition II.7.5]{hartshorne}),
hence the property of $mH + nL$ being ample only
depends on the fraction $m/n$,  so these notions are well-defined since for $n$ fixed and $m$ large (negatively large)
the divisor $mH+nL$ will be ample (antiample) by \cite[Exercise II.7.5 (b)]{hartshorne}. The linear combination
$ mH +nL$ is ample for $ \frac{m}{n} >  \at (L)$ and not ample for $\frac{m}{n} <  a(L) $.

In terms of $V=\nesenum (Y) \otimes_{\mathbb Z} {\mathbb R}$ and the ample cone, the ample threshold of $L$
is the number $\at$ where $\at H + L$ meets (the closure of) the ample cone. Recall that $\nesenum (Y)$
(called the \emph{numerical N\'{e}ron-Severi group} of $Y$; the terminology is not consistent in the literature)
is the Picard group of $Y$
modulo numerical equivalence, i.e. the equivalence relation where $D_1 \equiv D_2$ if and only if $D_1.C = D_2.C$ for all curves $C$,
and the ample cone is the convex cone spanned by all ample divisor classes,
see \cite[Section 1.4.C]{lazarsfeldpositivity1}. The ample cone is inside the (pseudo-)effective cone,
which is the cone generated by the pseudo-effective divisors,
i.e. divisors where some positive multiple is effective (see \cite[Section 2.2.B]{lazarsfeldpositivity1}). 

The ample threshold is always greater or equal to the antiample threshold. An ample divisor has negative
ample threshold and an antiample divisor has positive antiample threshold.
For $L= cH$ the ample threshold equals the antiample threshold and both are $-c$
(it leaves the antiample cone when it enters the ample cone).

For a vector bundles $\shG$ there are two related notions one can think of, namely
$ \operatorname{inf} \{ \frac{m}{p^e} :\,  (F^{e*}\shG)(m)  \text{ is ample}\}$ in positive characteristic
and $ \operatorname{inf} \{ \frac{m}{n} :\,  (S^{n}\shG)(m)  \text{ is ample}\}$ in general,
where $S^{n}$ denotes the symmetric power. We will not pursue this here.
\end{rmk}

The two concepts are related by $ \aat (L)= - \at (-L)$. We will work mainly with the antiample threshold.

For the next statement, recall that Kodaira vanishing for an ample line bundle $\shL$
means that $H^i ( Y, \shL \otimes \omega_Y)=0$ for $i \geq 1$, where $\omega_Y$ is the canonical sheaf on $Y$.
It holds in characteristic $0$ and in characteristic $p \geq \dim Y$, see \cite{deligneillusie}.
Kodaira vanishing and in fact even Kodaira-Ramanujam vanishing for big and numerically effective line bundles
hold for $K3$ surfaces in all characteristics, 
see  \cite{saintdonat} or \cite[Proposition 3.1]{huybrechtsk3}.
The property that the ample cone and the effective cone coincide up to closure is a bit artificial
but nevertheless justified by the existence of sufficiently many examples with this property.

\begin{lemma}
\label{antiamplethreshold}
Let $Y$ be a smooth projective polarized variety of dimension $d$ with fixed very ample divisor $H$
(with corresponding invertible sheaf $\O_Y(1)$).
Let $L$ be an antiample divisor  (with corresponding invertible sheaf $\shL$)
with antiample threshold $\aat=\aat(L)$. 
Suppose that the closure of the pseudoeffective cone equals the closure of the ample cone and that
Kodaira vanishing holds on $Y$.
Then
\[\sum_{m \in {\mathbb N} } h^d( \shL^n   (m)   ) 
=  n^{d+1}    \frac{\aat}{d! }    \sum_{i=0}^d  \frac{1}{i+1}  \binom{d}{i} \aat^i    H^i  .  L^{d-i}   
+ O(n^d) \, . \]
\end{lemma}
\begin{proof}
We fix $n$ and look at $  h^d( \shL^n (m) ) $.
By Serre duality we have
\[  H^d(Y,  \shL^n (m) ) \cong  H^0(Y, \shL^{-n} (-m) \otimes \omega_Y )^\dual       \]
and
\[  H^{i} (Y, \shL^n  (m) )  \cong  H^{d-i} (Y, \shL^{-n} (-m) \otimes \omega_Y )^\dual     \, .  \]
For $m < n\aat$ the invertible sheaf $\shL^{n} (m)$ is antiample and so $ \shL^{-n} (-m) $ is ample;
hence $H^{d-i}(Y, \shL^{-n} (-m) \otimes \omega_Y) =0$ 
by Kodaira vanishing (for $i \leq d-1$).
Therefore in this range we have by Riemann-Roch
\begin{eqnarray*}
h^d(\shL^n(m)) & = &\chi(\shL^n(m)) \\
&=& \frac{ (\shL^n (m))^d }{d!} +  O(n^{d-1}) \\
&=& \frac{ (mH +  n L )^d }{d!} +  O(n^{d-1} ) \\
&=& \frac{1}{d!} \left( \sum_{i=0}^d   \binom{d}{i}  m^i  n^{d-i} H^i  . L^{d-i}  \right)+  O(n^{d-1} )
\end{eqnarray*} 
Note that the Euler-characteristic $\chi (\shL^n(m))$ is by Riemann-Roch
(\cite[Corollary 15.2.1]{fultonintersection}, see also \cite[Corollary VI.2.14]{kollarrational})
a polynomial of degree $d$ in the two variables
$n$ and $m$ and that the coefficients depend only on the intersection behavior of $H$ and $L$ and data of the variety. 
This implies that the $O(n^{d-1})$-term above is a polynomial of degree $d-1$ in $n$ and $m$
and that summing them up for $m=0$ to $\lceil n \aat \rceil -1$ gives an $O(n^d)$-term.

In order to understand $h^d (\shL^n(m))$ for $m \geq n \aat$,
let $\ell \geq 0$ be such that $\O_Y(- \ell) \otimes \omega_Y$ is antiample.
Then for $m \geq  n \aat + \ell $ we have
\begin{eqnarray*}
h^d( \shL^n (m) )  &=& h^d(Y, \shL^n( m -\ell)  \otimes \O_Y(\ell)  ) \\
 &=&   h^0(Y, \shL^{-n} (-m +\ell) \otimes \O_Y(-\ell) \otimes \omega_Y )  \, .
\end{eqnarray*}
We know that $\shL^{-n} (-m+\ell)$ is not ample and so also $  \shL^{-n} (-m+\ell) \otimes \O_Y(-\ell) \otimes \omega_Y$
is not ample nor on the boundary of the ample cone.
Therefore there are no global sections by our assumption.
Finally, the range for $m$ between $\lceil n \aat \rceil $ and $ \lceil n \aat \rceil + \ell $ is finite (independent of $n$).
So the corresponding sum
$ \sum_{m=   \lceil n \aat \rceil  }^{ \lceil n \aat \rceil +\ell }   h^d(\shL^n(m))$
is bounded by a multiple of $n^d$.
Therefore we get (using that summation is integration up to an error of lower degree)
\begin{eqnarray*}
& & \sum_{m \in {\mathbb N}} h^d( \shL^n  (m) ) \\
&=&  \sum_{m =0}^{  \lceil n \aat \rceil -1  } h^d( \shL^n   (m)   )
+ \sum_{m = \lceil n\aat \rceil   }^{  \lceil n\aat \rceil +\ell   } h^d( \shL^n   (m)   ) 
+ \sum_{m = \lceil n\aat \rceil +\ell +1 }^{   \infty  } h^d( \shL^n   (m)   )   \\
&=&  \sum_{m =0}^{  \lceil n\aat \rceil -1  } \frac{1}{d!} 
\left( \sum_{i=0}^d   \binom{d}{i}  m^i  n^{d-i} H^i .  L^{d-i}   \right) + O(n^d)  \\   
&=&  \frac{1}{d!}  \sum_{i=0}^d  \binom{d}{i}  n^{d-i}
\left(       \sum_{m =0}^{  \lceil n\aat \rceil -1  }    m^i   H^i . L^{d-i}  \right)
+ O(n^d) \\
&=&   \frac{1}{d!}  \sum_{i=0}^d  \binom{d}{i}  n^{d-i}
\left(     \frac{1}{i+1}    (n \aat)^{i+1}   H^i .  L^{d-i}  \right)
+ O(n^d)                      \\
&=&  \frac{\aat}{d!}   n^{d+1}         \left(  \sum_{i=0}^d  \frac{1}{i+1}  \binom{d}{i}  \aat^i   H^i .  L^{d-i}     \right)
+ O(n^d) \, .  \\   
\end{eqnarray*}
\end{proof}

\begin{cor}
\label{topasymptotic}
In the situation of Lemma \ref{antiamplethreshold} we get the asymptotic behavior
\[\lim_{n \rightarrow \infty}    \frac{ \sum_{m \in {\mathbb N}} h^d( \shL^n   (m)   )  }{n^{d+1} } = 
  \frac{\aat}{d!}      \left(  \sum_{i=0}^d  \frac{1}{i+1}  \binom{d}{i}  \aat^i    H^i  .   L^{d-i}    \right) \, .\]
\end{cor}
\begin{proof}
This follows directly from Lemma \ref{antiamplethreshold}.
\end{proof}

For $d=2$ this limit is
\[    \frac{\aat}{2} \left( \frac{1}{3} \aat^2 H^2   +  \aat H.L   +   L^2    \right)    \, .\]
In the following we restrict to the case of a smooth projective surface $S$.
In this case $V= \nesenum (S) \otimes_\Z {\mathbb R}$
carries an integral intersection  bilinear form and in particular an integral quadratic form
(integral means that the values on the underlying integral lattice are integers).
The positive doublecone is given by the locus where the quadratic form is
non-negative. It is symmetric to the origin and consists of two convex cones,
one being determined by containing the ample class, which we just call the positive cone.

\begin{lemma}
\label{ampleeffectivepositive}
Let $S$ be a smooth projective surface over an algebraically closed field $K$.
Suppose that the closure of the ample cone equals the closure of the effective cone in
$\nesenum (S) \otimes_\Z {\mathbb R}$.
Then this cone equals also the positive cone. 
\end{lemma}
\begin{proof}
The ample cone is always inside the positive cone.
So suppose that $D$ is a divisor with positive self intersection number $D^2 > 0$ and that $D$ lives in the ample half. 
The Hodge index theorem \cite[Theorem V.1.9]{hartshorne} excludes $H.D=0$. Then $H.D >0$ ($H.D <0$ gives a contradiction) and
so by \cite[Corollary V.1.8]{hartshorne}
we have that $nD$ is effective for some $n \in \N$. It follows that $D$ is in the closure of the ample cone by assumption. 
\end{proof}

\begin{example}
We consider a product $\surface = C \times {\mathbb P}^1_K $
where $C$ is a smooth projective curve over an algebraically closed field $K$.
The numerical N\'{e}ron-Severi group is $\Z \times \Z$, the standard vectors being represented by the fibers $E$ and $F$
of the two projections.
A divisor $rE+sF$ is effective if and only if $r,s \geq 0$ and ample if and only if $r,s > 0$.
Hence the effective cone equals the ample cone up to closure.
For an ample divisor $H = rE + sF$ and a divisor $D= m E + n F$, the antiample threshold of $D$ (with respect to $H$) is
$\aat(D) = \operatorname{min} (  \frac{-m}{r} , \frac{-n}{s})  $. In particular it is a rational number.
The same behavior holds for ruled surfaces if the defining bundle of rank two
is strongly semistable, see \cite[Section 1.5.A]{lazarsfeldpositivity1}. It follows from Corollary \ref{topasymptotic} that for any line bundle $\shL$
on such a ruled surface the limit
$\frac{ \sum_{m=0}^\infty \dim_K H^2 (\surface, \shL^q (m)) }{q^3}$ is a rational number.
\end{example}

We are now in a state to use irrational boundaries of the ample cone to produce asymptotic behavior of cohomology of line bundles
with irrational limits. This method was pioneered by D. Cutkosky
(see \cite{cutkoskyzariski}, \cite{cutkoskysrinivas}, \cite{cutkoskyirrational}
and \cite[Section 2.3.B]{lazarsfeldpositivity1}).
Note that if $\aat$ is an irrational square root of a rational number,
then the expression $ \frac{\aat}{2} \left( \frac{1}{3} \aat^2 H^2   +  \aat H.L   +   L^2    \right) $
is irrational as well under the conditions of the following lemma.
To see this we only have to exclude that
$\frac{1}{3} \aat^2 H^2 + L^2       = 0$.
Since $(\aat H +L)^2 =\aat^2 H^2     +2 \aat L.H +   L^2       =0$,
we would have $2  \aat H.( \frac{1}{3} \aat H + L) =0  $.
But $ \frac{1}{3} \aat H +L$ is antiample and can not have $0$ intersection
with $H$.

\begin{cor}
\label{sheafcohomologyirrational}
Let $\surface$ be a smooth projective surface over an algebraically closed field $K$ of positive characteristic $p$
with a fixed very ample invertible sheaf $\O_\surface(1)$ with corresponding divisor $H$.
Suppose that the closure of the ample cone equals the closure of the effective cone,
that this cone restricted to an integral plane containing $H$ has irrational boundaries
and that Kodaira vanishing holds on $\surface$.
Then  there exists an antiample line bundle $\shL$ whose antiample threshold $\aat$ is an irrational square root of a rational number
and we have the irrational limit ($q=p^e$)
\[ \lim_{e \rightarrow \infty} \frac{ \sum_{m \in {\mathbb N}} h^2( \shL^{q}    (m)   )  }{ q^{3} } 
=  \frac{\aat}{2} \left( \frac{1}{3}  \aat^2 H^2   +  \aat L.H     +       L^2              \right)   \, .   \]
Moreover, there exists a line bundle $\shM$ (one can take $\shM=\shL$) such that
\[ \lim_{e \rightarrow \infty} \frac{ \sum_{m \in {\mathbb N}} h^1( \shM^{q}    (m)   )  }{ q^{3} }     \]
and
\[ \lim_{e \rightarrow \infty} \frac{ \sum_{m \in {\mathbb Z}} h^1( \shM^{q}    (m)   )  }{ q^{3} }     \]
are irrational.
\end{cor}
\begin{proof}
Under the given assumptions,
the ample cone in $\nesenum (\surface)_\R $ is by Lemma \ref{ampleeffectivepositive} given by one half (the half containing $H$)
of the positive doublecone of an integral quadratic form
$n_1x_1^2-n_2x^2_2 - \cdots - n_s x^2_s$
where $n_i \in \N$ (the type of this form is determined by the Hodge index theorem, see \cite[Theorem V.1.9]{hartshorne})
and where the first component corresponds to $H$.
The condition means that there exists a plane spanned by $H$ and an integral divisor
$D$ (or rather its class in $\nesenum (\surface)$) such that $uH + D $
meets the boundary of the ample cone in an irrational number $u$.
Replacing $D$ by a suitable integral combination
$rH+sD$ we may assume $H.D=0$
(note however that $H$ and $rH+sD$ do not generate the same sublattice as $H$ and $D$, only the same real plane).
Then $(uH+D).(uH+D) = u^2 H^2 +  D^2        = 0$ 
shows that $u^2 $ is rational.
The intersection form is now symmetric to the $H$-axis and the $D$-axis.
Now choose $L= c H + D$ to be antiample and denote the corresponding invertible sheaf by $\shL$.
Then the result follows from Corollary \ref{topasymptotic}.

For the second statement we take
the invertible sheaf $\shM$ corresponding to $D$ to get a clearer view on the two distinct boundary phenomena.
Since $H.D=0$ we must have $D.D < 0$
by the Hodge index theorem.
The same applies to $-D$. Hence $D$ lies outside the ample and the antiample cone,
and the ample threshold is $u$ and the antiample threshold is $-u$ by the symmetry of the situation. 

Let $\ell \geq 0$ be such that $\O_\surface(-\ell) \otimes \omega_\surface$ is antiample. Then for $m-\ell > n u$
we have
$H^1(\surface,\shM^n(m)) = H^1(\surface, \shM^n(m-\ell) \otimes \O_\surface(\ell)
\otimes \omega_\surface^{-1} \otimes \omega_\surface)$,
and this is $0$ since
$\shM^n(m - \ell) \otimes \O_\surface(\ell)$ and $ \O_\surface(\ell) \otimes \omega_\surface^{-1} $
are ample and by Kodaira vanishing.
For $ \frac{m}{n} <u$, the sheaf $ \shM^n (m) $ has no non-zero sections.
Also, $H^2(\surface, \shM^n (m)) = 0$ for $m \geq 0$,
since $ \shM^{-n} (-m) \otimes \omega_\surface   $ also has (for $n$ large enough)
negative self intersection and is therefore not effective. Hence we get
\newpage
\begin{eqnarray*}
& & \sum_{m \in {\mathbb N}} h^1( \shM^{n}    (m)   )  \\
& =&  \sum_{m =0}^{ \lceil u n \rceil   -1 } h^1( \shM^{n}    (m)   ) 
+ \sum_{m= \lceil u n \rceil}^{  \lceil u n \rceil + \ell        }    h^1(\shM^n(m)) 
 + \sum_{m= \lceil u n \rceil+\ell+1}^\infty h^1(\shM^n(m)) 
\\
&=& - \sum_{m =0}^{ \lceil u n \rceil -1 } \frac{(mH +nD)^2}{2} +  O(n^2)  \\
&=&  - \frac{1}{2}  \sum_{m =0}^{  \lceil u n \rceil   -1 }  ( m^2 H^2 +n^2 D^2  )  +   O(n^2)   \\
&=&  - \frac{1}{2} \left( \frac{1}{3} u^3 n^3 H^2+un^3 D^2   \right)     +   O(n^2)    \\
&=& - \frac{ u}{2} n^3 \left( \frac{u^2 H^2}{3} +  D^2 \right)  +  O(n^2) \, .
\end{eqnarray*}
Dividing by $n^3$ the additional term on the right converges to $0$. Note that since $\frac{u^2 H^2}{3} +D^2 \neq 0$,
the limit is an irrational number.
The same limit behavior
holds for $q=p^e$ and $e \rightarrow \infty$ instead of $n$.

Similar arguments show that for $\! m  \leq 0 \!$ we only have to consider \linebreak
$h^1(\shM^n(m))$ in the range $ -un \leq m \leq 0$, and that in this range  $h^0(\shM^n(m))$ and $h^2(\shM^n(m))$ are $0$.
The summation gives in the limit again
a positive rational multiple of $u$. For $L=c H+D$ antiample we have
\[ \sum_{m \in \Z}  h^1(\shL^n(m) = \sum_{m \in \Z}  h^1(\shM^n(m) =   \sum_{m \in \N}  h^1(\shL^n(m) \, , \]
so we can also take $\shL$ instead of $\shM$.
\end{proof}

\begin{rmk}
\label{sublattice}
The proof of Corollary \ref{sheafcohomologyirrational} shows that in order to get the stated limit behavior
it is enough to have a sublattice of rank two inside
$\nesenum ( \surface)$ which contains $H$ and
such that the induced ample cone and the induced effective cone are equal up to closure
and such that their boundary is irrational.
However, in general it depends on the full $\nesenum (\surface)$ whether a divisor is ample or not. 
\end{rmk}

\section{Hilbert-Kunz multiplicity for splitting top-dimensional syzygy bundle}
\label{splitting}

We turn now our attention to situations where the top-dimensional syzygy bundle splits into line bundles.

\begin{thm}
\label{hilbertkunzsplitting}
Let $R$ be a standard-graded Cohen Macaulay domain with an isolated singularity of dimension $d +1 \geq 2$
over an algebraically closed field of positive characteristic $p$ and let
$M$ be a graded $R$-module of finite length. Let
\[ \cdots   \longrightarrow F_2       \longrightarrow F_1 
\longrightarrow  F_0  \longrightarrow M \longrightarrow 0 \]
be a graded complex which is exact on $D(R_+)$, where $F_i =\bigoplus_{j \in J_i}   R(- \beta_{i j})$ are graded free $R$-modules.
Set $Y= \operatorname{Proj} R$ and assume that in $\nesenum (Y)_\R $ the closure of the ample cone and the closure of the (pseudo)effective cone are identical and that Kodaira vanishing
holds.
Suppose that $\operatorname{Syz}_{d} \cong \shL_1 \oplus \cdots \oplus \shL_r$
is a direct sum of line bundles with antiample thresholds $\aat_i=\aat(L_i)$.
Then
\begin{eqnarray*}
& & \operatorname{HK}(M)   \\
& \!\!\! \! \! = \!\!\!\! \!\! &\sum_{j=1}^r \! \left(\! \frac{\aat_j}{d!} 
\sum_{i=0}^d \frac{1}{i+1}  \binom{d}{i}  \aat_j^i   H^i . L_j^{d-i} \!  \right)\!
+ \! \frac{H^d}{(d+1)!}   \!\!\!    \left(    \sum_{i=0}^{d} (-1)^{ d+1-i } \!\!\! \left(    \sum_{j \in J_i } 
   \beta^{d+1}_{ij}  \right) \!\!\! \right) \!\!   . 
\end{eqnarray*}
\end{thm}
\begin{proof}
This follows from Theorem \ref{hilbertkunzgeometric} and Corollary \ref{topasymptotic}.
\end{proof}

\begin{rmk}
It is not easy to give examples where the splitting behavior supposed in Theorem \ref{hilbertkunzsplitting}
actually holds for a top-dimensional syzygy bundle. It holds when $M$ has finite projective dimension.
See Example \ref{quadric} below
for a splitting case where the projective dimension is not finite. This behavior seems to happen often in the toric case.
Splitting criteria in various situations were given in \cite{bakhtarysplitting}, \cite{halicsplit}, \cite{horrockspunctured},
\cite{mohankumarraoravindra}, \cite{sawadasplitting}. However,
these results mean often splitting into line bundles of the form $\O_Y(\ell)$.
\end{rmk}

\begin{rmk}
If the top-dimensional syzygy bundle splits,
but the effective cone does not coincide with the ample cone up to closure, then two boundaries
are important, the number $\aat$ where $\aat H +D$ leaves the antiample cone and the number
$c$ where $c H +D$ leaves the antieffective cone (so $c$ is the \emph{antieffective threshold}).

Recall that the effective cone equals the big cone up to closure (see \cite[Theorem 2.2.26]{lazarsfeldpositivity1}.
For $m,n$ such that $ \aat \leq   \frac{m}{n} \leq c$ the invertible sheaves $\shL^n(m)$
are not antiample, but antibig.
Hence $h^d (\shL^n(m)) $ can neither be computed numerically by Riemann-Roch,
since there will be intermediate cohomology, nor is it $0$. The sum of this cohomology over this range,
divided by $n^{d+1} $, will in fact contribute substantially, since by antibigness $h^d (\shL^n(m))$ is large of order $v n^d$,
where $v$ stands for the volume of $\shL^{-n}(-m)$ (for a fixed ratio $m/n$).
The exact value can probably be obtained by integrating the continuous volume function
(\cite[Corollary 2.2.45]{lazarsfeldpositivity1}, see also  Cutkosky's positive intersection product
\cite[Theorem 5.2, Theorem 5.3]{cutkoskyteissierproblem}) in this range.

We also conjecture that there exist examples of this type where the Hilbert-Kunz multiplicity is a transcendental number,
and where also the second coefficient in the Hilbert-Kunz function
(see \cite{hunekemcdermottmonsky}) is irrational
(according to the possible irrationality of the volume of big divisors, see \cite[Section 2.3.B]{lazarsfeldpositivity1}).
We will work this out in the near future.
P.Monsky has shown in \cite{monskytranscendence} that if his conjecture
on a specific plane curve of degree three in characteristic two
holds, then he gets (not only the existence of irrational Hilbert-Kunz multiplicities as mentioned in the introduction but also)
that various explicit transcendental special values of certain hypergeometric functions are  
$\Q$-linear combinations of Hilbert-Kunz multiplicities of characteristic $2$ local rings.
\end{rmk}

\begin{rmk}
By the filtration principle for vector bundles (usually called splitting principle,
see \cite[Appendix A.3]{hartshorne} or \cite[Section 3.2]{fultonintersection}),
there exists always a birational morphism
$\psi:Y' \rightarrow Y$ with $Y'$ smooth and projective
such that $ \psi^{*} (E)$ has a filtration $E_1 \subset E_2 \subset \ldots \subset E_{r-1} \subset
\psi^{*} (E) $ with invertible quotients  $E_{i+1}/E_{i}$. These can be constructed inductively by looking at
the projective bundle ${\mathbb P}(E)$
and suitable subvarieties and their resolution of singularities.
Moreover, K. E. Sumihiro has shown in \cite[Theorem 2.1]{sumihiro}
that there exists a finite flat morphism $Y' \rightarrow Y$ with $Y'$
smooth and projective and such that the pull-back of $E$ has a filtration with invertible quotients.
\end{rmk}

\begin{example}
\label{quadric}
We consider a smooth quadric in four variables over an algebraically closed field of
characteristic $p$. They are all isomorphic and 
might be given by $XY-ZW=0$ or by $X^2+Y^2+Z^2+W^2=0$ (for $p \neq 2$). The corresponding projective surface $S$ is isomorphic to
${\mathbb P}^1 \times {\mathbb P}^1$. Let $p_1$, $p_2$ denote the projections and set
$ \shM = p_1^*( \O_{\mathbb P^1} (1))$ and
$ \shN = p_2^*( \O_{\mathbb P^1} (1))$. On the surface we have
\[\O_S(1) = \shM \otimes \shN \text{ and } \Omega_S =  \shM^{-2} \oplus \shN^{-2} \, . \]
On the projective line ${\mathbb P}^1= \Proj K[s,t]$ we have the short exact sequence
\[ 0 \longrightarrow  \O_{\mathbb P^1}(-2) \longrightarrow   \O_{\mathbb P^1}^2     \stackrel{s^2,t^2}{\longrightarrow}  \O_{\mathbb P^1}(2)  \longrightarrow 0 \, .\]
This pulls back to
\[ 0 \longrightarrow  \shM^{-2} \longrightarrow   \O_S^2     \longrightarrow  \shM^2  \longrightarrow 0 \, \]
and so we get
\[ 0 \longrightarrow  \shM^{-2} \oplus \shN^{-2} \longrightarrow   \O_S^4 
\longrightarrow  \shM^2 \oplus \shN^2 \longrightarrow 0 \, . \]
Tensoring with $\O_S(-2)= \shM^{-2} \otimes \shN^{-2}$ yields
\[ 0 \longrightarrow  \left( \shM^{-4} \otimes \shN^{-2} \right) \oplus \left( \shM^{-2} \otimes  \shN^{-4} \right)
\longrightarrow   \O_S (-2)^4
\longrightarrow  \shM^{-2} \oplus \shN^{-2} \longrightarrow 0 . \]
Here $\shE \cong  \left( \shM^{-4} \otimes \shN^{-2} \right) \oplus \left( \shM^{-2} \otimes  \shN^{-4} \right) $
in the notation of Lemma \ref{hypersurfaceresolution}.
The antiample threshold on these two line bundles is $2$. The self intersection of $ \shM^{-4} \otimes \shN^{-2}$ is $16$ and
its intersection with $H$
is $-6$ (the same holds for $ \shM^{-2} \otimes \shN^{-4}$). Therefore the Hilbert-Kunz multiplicity is
 \[ 2    \left(  \frac{1}{3} 8   + 2   (-6)   + 16 \right)  +   \frac{2}{6} ( -5 \cdot 8 +  4    )
 =  \frac{4}{3} \, . \]
This coincides with the known result, see \cite[Theorem 3.1]{watanabeyoshida3}.
\end{example}

The following result is known at least for $M=R/\m$ by work of K. Watanabe \cite{watanabetoric} on normal monoid rings and follows also
from \cite{seibertfinitetype} in connection with \cite{smithbergh}.
It is probably also possible to prove it by other methods like  \cite{conca} or \cite{millerswanson}.

\begin{cor}
\label{quadricrational}
Let $R=K[X,Y,Z,W]/(F)$ be given by a quadratic equation
over an algebraically closed field $K$
such that $\surface =\Proj R$ is smooth.
Let $M$ be a graded $R$-module of finite length. Then the Hilbert-Kunz multiplicity of $M$ is a rational number.
\end{cor}
\begin{proof}
Let 
\[ 0 \longrightarrow \Syz_2 \longrightarrow F_2 \longrightarrow F_1 \longrightarrow F_0 \longrightarrow M \longrightarrow 0 \]
be the truncated minimal free resolution of $M$. Then $\Syz_2$ is a maximal Cohen-Macaulay module.
The same is true for the corresponding
situation on $S$. By a standard classification result it is known that the only
maximal Cohen-Macaulay modules on quadrics are spinor bundles and certain line bundles,
and that in the given dimension the spinor bundles are
themselves line bundles
(see \cite{ottavianispinor},  \cite{knoerrerperiodicity}). Hence we have $\Syz_2 \cong \shL_1 \oplus \cdots \oplus \shL_r$.
So the result follows from Theorem \ref{hilbertkunzsplitting} and the shape of the intersection form on
$\nesenum (\surface) = \operatorname{Pic} (\surface)= \Z^2$ (see \cite[Example V.1.10.1]{hartshorne}).
\end{proof}

\section{Determinantal quartics}
\label{determinantalquartic}

In this section we describe determinantal quartic surfaces following \linebreak \cite{beauvilledeterminantal},
which give explicit examples of smooth projective surfaces
where the ample cone equals the effective cone up to closure (at least on certain subplanes inside the N\'{e}ron-Severi group)
and where the boundary of the ample cone is irrational.
This part is heavily influenced by the papers \cite{oguisoentropy} of K. Oguiso and
\cite{festigarbagnatigeemenluijk} by D. Festi, A. Garbagnati, B. van Geemen and R. van Luijk,
see in particular \cite[Lemma 3.1, Corollary 3.4, Theorem 4.1, Remark 4.2]{oguisoentropy} and
\cite[Theorem 1.2, Proposition 2.2, Theorem 4.5]{festigarbagnatigeemenluijk}.
I am especially grateful to R. van Luijk for explaining
several aspects in positive characteristics of \cite{festigarbagnatigeemenluijk} to me.

A determinantal quartic in four variables is given by the determinant $F= \operatorname{det} A$ of a matrix
\[ A   =   (  L_{ij} )_{1 \leq i,j \leq 4}  \, ,  \]
where the $L_{i j}$ are linear polynomials in $P=K[X,Y,Z,W]$. Hence a determinantal quartic is a homogeneous form of degree four. It defines
a projective surface $V_+(F) \subset {\mathbb P}^3$ of degree four, hence it is a $K3$ surface (provided it  is smooth)
and its canonical class is trivial. Further properties like smoothness and the shape
of the Picard group and the intersection form depend on the linear entries of the matrix.

\begin{lemma}
\label{determinantalbasic}
A smooth determinantal quartic surface $S$ contains a smooth curve $C$ of genus $3$ and degree $6$.
If $H$ denotes the ample class corresponding to $\O_\surface(1)$, then the intersection form is given by
$H^2=4$, $H.C= 6$, $C^2=4$ on the plane spanned by $H$ and $C$ inside the real N\'{e}ron-Severi group $\nesenum (S)_\R$.
The positive doublecone has irrational boundaries, there are no
integral (or $\Q$-) curves with self intersection $-2$ or $0$ and
the effective cone equals the positive cone up to closure. 
If the Picard rank of $S$ is two, then the ample cone equals the effective cone up to closure.
\end{lemma}
\begin{proof}
The first statement is \cite[Corollary 6.6]{beauvilledeterminantal}.
The intersection numbers follow from this and the adjunction formula
$C^2 = C.(C+K_S) =  2g-2$ (see \cite[Proposition V.1.5]{hartshorne}).
Setting $D=2H-C$ the intersection form on this plane is given by
\[H^2=4,\, H.D= H.(2H-C)= 2 \text{ and }  D^2 = (2H-C)^2 = -4 \, . \]
Thus, in the plane spanned by $H$ and $D$ inside $\nesenum (S)_\R$, the quadratic form is given by
$4x^2 +4xy - 4y^2 $, where $x$ and $y$ denote the coordinates for $H$ and $D$.
The boundary of the positive doublecone is hence given by 
\[ y = \frac{\pm \sqrt{5} + 1 }{2} x \]
in these coordinates, so the slopes are given by the golden ratio.
From this it follows that non-zero integral or $\Q$-divisors in this plane do not have zero self intersection.
If we write $x= n_1/m$ and $y=n_2/m$ as rational numbers, then the possible values of the quadratic form are
\[ \frac{4}{m^2}  (n_1^2 +n_1n_2 -n_2^2 )\]
The integral equation
\[ 4   (n_1^2 +n_1 n_2 -n_2^2 )  = c m^2    \]
can only have a solution when the exponent of $2$ in $c$ is even. This follows immediately
from looking at the possible parities
of $n_1$, $n_2$.
So in particular the quadratic form can not have the value $-2$ (not even for rational arguments).
Other negative integral values of the quadratic form of a $K3$ surface are excluded by the adjunction formula,
so the form is positive on the effective cone and therefore the effective cone lies inside the positive cone.
As the other inclusion (up to closure) always holds (\cite[Corollary V.1.8]{hartshorne}), these two cones coincide.

If the Picard rank is two, then the complete intersection behavior is encoded in our plane, and so there are no $0$- or
$-2$-curves at all
on the surface. Hence every effective divisor is numerically effective and these form the closure of the ample cone.
\end{proof}

\begin{rmk}
\label{idealminors}
Following \cite[6.7]{beauvilledeterminantal} (see also \cite[Proposition 2.5]{festigarbagnatigeemenluijk}),
the ideal corresponding to the curve $C$ can be made more explicitly.
If we delete in the matrix
$A$ the first column and call the remaining $4 \times 3$-matrix $B$, then we have an exact sequence
\[0 \longrightarrow P (-4)^3 \stackrel{B}{\longrightarrow } P (-3)^4 \longrightarrow P \longrightarrow P/I \longrightarrow 0 \, , \]
where the homomorphism in the middle is given by the maximal minors of $B$. The exactness is a direct consequence of linear algebra,
applied to $B$. By computing the determinant of $A$ using the first row,
we also see again that the curve $C=V_+(I)$ lies on the surface
$V_+(F)$. The ideal $I$ has height two in the polynomial ring and height one in $R=P/(F)$.
\end{rmk}

\begin{example}
Determinantal equations are getting quickly quite long and complicated unless they define a singular variety.
The following example was
found by D. Brinkmann with the help of \cite{M2}. The determinant of
\[ A = \begin{pmatrix}
X & Y  &  Z  & 0   \\
Y & Z   & 0  & W   \\
Z  & 0 &  W  & X   \\
W & W  & X  & Y   
 \end{pmatrix}   \, 
\]
defines a projective variety which is smooth unless the characteristic is $p=37013$, $651881$, $742991$. The determinant is
\[- XYZW  -X^3Z -Y^3W -XW^3 -YZ^3+ YW^3 +X^2Y^2 +Z^2W^2 +XZ^2W \, . \]
The ideal defining the curve $C$ is
$I= (-X^2Z + YZW-W^3, \,Y^2W-X^2Y+XZW, \, -XYW-YZ^2 + ZW^2,\, -YW^2-XZ^2 )$ according to
Remark \ref{idealminors}.
\end{example}

The outcome so far is that whenever we can establishe a smooth determinantal quartic surface with Picard rank two, then
we can apply Lemma \ref{determinantalbasic} and  Corollary \ref{sheafcohomologyirrational}
to produce examples of line bundles with irrational Frobenius-asymptotical behavior. To achieve this,
we will now work with more specific determinantal quartics, so that the Picard rank is two and that there
exists nontrivial automorphisms.
In \cite{oguisoentropy}, the author gave an example of a $K3$ surface $S$ over $\mathbb C$
where the automorphism group is large in the sense that there exists a fixpoint free automorphism
such that the corresponding homomorphism on the second singular cohomology $H^2 (S, {\mathbb C})$
has an eigenvalue whose absolute value is larger than $1$.
In \cite{festigarbagnatigeemenluijk}, the authors established a relationship between this example and
work of Cayley \cite{cayleymemoir} and reinterpreted it in terms of determinantal quartic surfaces.
To establish that there exists such surfaces with Picard rank two, they looked at the following example.

\begin{example}
\label{fgglmatrix}
Consider the matrix
\[ A = \begin{pmatrix}
X & Z  &  Y+Z  & Z+W   \\
Y & Z+W  & X+ Y+Z+W  & X+W   \\
X+Z  & X+Y+ Z+W  &  X+Y  & Z   \\
X+Y+W & X+ Z  & W  & Z   
 \end{pmatrix}   \, .
\]
The surface $\surface = V_+(F)$,  where $F= \operatorname{det} (A)$, is smooth in characteristic zero and in characteristic $2$ by
\cite[Theorem 4.5]{festigarbagnatigeemenluijk}, in fact it has singularities exactly in characteristics
$3$, $5$, $7$, $13$, $ 443$, $5399$, $9562057$, $578193147733$, $2202537665175172539619840469$
(this was been checked with the help of Macaulay 2 \cite{M2} by D. Brink\-mann).
With a careful analysis of the situation in characteristic two it was shown in
\cite[Theorem 4.5]{festigarbagnatigeemenluijk} that
the Picard rank for this equation in characteristic two is two.
From this they deduce that the Picard rank in characteristic zero is also two,
using the fact that for a variety over a number field the N\'{e}ron-Severi group
of the fiber in characteristic zero embeds into the N\'{e}ron-Severi group of any smooth special
fiber in positive characteristic (see \cite[Example 20.3.6]{fultonintersection})
and that the Picard group is the N\'{e}ron-Severi group for hypersurfaces of dimension $\geq 2$.
\end{example}

From this example it follows that for a general determinantal quartic surface in characteristic zero the Picard rank is also two.

The following lemma shows that in order to get irrational behavior of the Frobenius asymptotic for almost all prime characteristics
it is enough to have Picard rank two in characteristic zero.

\begin{lemma}
\label{ampleeffectivesublattice}
Let $F$ be a determinantal equation of degree four in four variables over $\Z$ and suppose that the surface $V_+(F)$ is smooth
with Picard rank two in characteristic zero (over $\overline{\mathbb Q}$).
Then for almost all prime numbers,
the effective cone equals the ample cone inside the $H-D$-plane of $\nesenum (S_p)$  up to closure
($H$ and $D$ as in Lemma \ref{determinantalbasic}).
\end{lemma}
\begin{proof}
We will work with the main result of \cite[Theorem 4.1]{oguisoentropy} and \cite[Theorem 1.2]{festigarbagnatigeemenluijk}, namely
that in this situation there exists a fixpoint free automorphism on the surface $S_{\overline {\mathbb \Q}}$
with rather special properties.
This automorphism is best understood by looking at the induced homomorphism on the Picard group $\nesenum (\surface) \cong \Z^2$.
The matrix $M =\begin{pmatrix}  1 &1 \\ 1 & 2 \end{pmatrix}$ has the property that its $n$th iteration is
$M^n = \begin{pmatrix}  f_{2n-2}  & f_{2n-1}  \\ f_{2n-1} & f_{2n} \end{pmatrix}$, where $f_k$ is the $k$th Fibonacci number.
Multiplication with this matrix defines an isomorphism of $\Z^2$ which respects the quadratic form
$4(x^2 +xy-y^2)$. Now the algebraic automorphism on $S$ induces the homomorphism $M^3$ on the Picard group by
\cite[Theorem 1.2]{festigarbagnatigeemenluijk}.

The automorphism can be defined over an algebraic extension of  $\Z$ after inverting a natural number.
Hence for almost all prime numbers the special surface $S_{\overline{ \kappa(p)}}$ in positive characteristic $p$
(is again smooth and) has an automorphism
which acts on the divisors $H$ and $D$ as in characteristic zero
(we do not know how it acts on the Picard group in case this group has higher rank for a specific prime number)
and the induced homomorphism is the same. As this homomorphism stems from an algebraic automorphism, it must respect the ample cone.
The ample cone in positive characteristic is contained inside the positive cone.
The ample cone contains $H$ and all
its images under the various automorphisms. These images are $ f_{2n-2} H + f_{2n-1} D $ (where $n$ is a multiple of $3$).
The rays given by these points approximate the upper boundary of the positive cone arbitrarily good. The inverse matrix of $M$ is
$M^{-1} = \begin{pmatrix} 2 & -1 \\ -1 & 1 \end{pmatrix}$, and the images of $H$ under these iteratives approximate the
lower boundary of the positive cone arbitrarily good. Therefore the ample cone equals the positive cone and by
Lemma \ref{determinantalbasic} also the effective cone up to closure.
\end{proof}

\begin{rmk}
Note that in the previous 
Lemma \ref{ampleeffectivesublattice} it is essential  to argue with the help of the automorphism. Of course,
an ample divisor $rH+sD$ in characteristic zero will be ample for almost all prime characteristics by the openness of ampleness
(see \cite[Theorem 1.2.17]{lazarsfeldpositivity1}).

.
However, the bound on the prime numbers depend on the divisor itself,
and so we can not exclude that the ample cone in the $H-D$-plane
is in all characteristics strictly smaller than the positive cone. From the openness of ampleness we can only deduce that the
ample cone converges to the positive cone as $p \rightarrow \infty$.
Anyway, this property is strong enough to establish at least that the limit of Hilbert-Kunz
multiplicities for $p \rightarrow \infty$ is irrational.
\end{rmk}

\begin{cor}
\label{oguisoirrational}
Let $\surface= V_+(\operatorname{det} A)$ be a determinantal quartic surface
defined over $\Z$ which is smooth and has Picard rank two in characteristic zero.
Then for almost all prime reductions there exist invertible sheaves $\shL$ and $\shM$ such that
the  limits ($\aat$ denotes the antiample threshold of $\shL$)
\[ \lim_{e \rightarrow \infty} \frac{ \sum_{m \in {\mathbb N}} h^2( \shL^{q}    (m)   )  }{ q^{3} } 
=  \frac{\aat}{2} \left(   \frac{1}{3}  \aat^2 H^2  +  \aat L.H   +  L^2    \right)   \, ,   \]
\[ \lim_{e \rightarrow \infty} \frac{ \sum_{m \in {\mathbb N}} h^1( \shM^{q}    (m)   )  }{ q^{3} }    \]
and
\[ \lim_{e \rightarrow \infty} \frac{ \sum_{m \in {\mathbb Z}} h^1( \shM^{q}    (m)   )  }{ q^{3} }   \,    \]
are irrational numbers. Moreover, these limits are independent of the characteristic.
\end{cor}
\begin{proof}
This follows from Corollary \ref{sheafcohomologyirrational} (taking Remark \ref{sublattice} into account) and 
Lemma \ref{ampleeffectivesublattice}. The last part follows from the construction of the
family and the computation in Corollary \ref{ampleeffectivepositive} shows that these numbers
only depend on intersection properties of $\shL$ and $\shM$
which are constant in the family.
\end{proof}

\begin{rmk}
For characteristic two, the result of Corollary \ref{oguisoirrational} follows directly from Lemma \ref{determinantalbasic} and the proof of
\cite[Theorem 4.5]{festigarbagnatigeemenluijk} where the authors show for the special matrix mentioned in Example \ref{fgglmatrix}
above that its Picard rank is two. Also, as van Luijk has pointed out, these explicit calculations
in characteristic two show that the Tate conjecture holds for the given surface.
It follows that we are in the situation of \cite[Theorem 1 (1)]{charlesk3} and hence there are infinitely many
prime numbers such that the Picard number of the special fiber is also two, hence we obtain the result for infinitely many
prime numbers also from this.
\end{rmk}

In the next section we will also need the following lemma.

\begin{lemma}
\label{oguisoh1}
Let $\surface =V_+( \det (A)) $ be a determinantal quartic surface
defined over $\Z$ which is smooth and has Picard rank two in characteristic zero. Let
$H$ be the ample class (corresponding to $\O_\surface(1)$)
and let $D$ be the class satisfying $H^2=4$, $H.D=2$, $D^2=-4$.
Let $\shM$ denote the invertible sheaf corresponding to $D$.
Then  $H^1(S,\shM(m))= 0$ for all $m \in \Z$ for almost all prime numbers.
\end{lemma}
\begin{proof}
Suppose that the surface is smooth in characteristic $p$.
We consider the self intersection number
\[ (mH+D)^2 = 4 m^2 +4m  -4 = 4(m^2+m-1) \]
which is positive unless $m=0,-1$.
Hence $\shM(m)$ will either be ample or antiample by Lemma \ref{ampleeffectivesublattice}
for $m \neq 0,-1$. In these cases
$H^1(S, \shM(m))=0$
by Kodaira vanishing (which holds on a $K3$ surface, see \cite[Proposition 3.1]{huybrechtsk3}) and since $\omega_S= \O_S$.
For $m=0,-1$ we need to take a closer look at Riemann-Roch (see \cite[Theorem V.1.6]{hartshorne} for the surface case).
First note that the self intersection number in both remaining cases is $-4$. On a $K3$ surface
we have $h^0(\O_S)=1$, $h^1(\O_S)=0$ and $h^2(\O_S)=1$ and hence $\chi (\O_S) =2$. Therefore by Riemann-Roch we have
\begin{eqnarray*}
 \chi( \shM(m)) & = &  h^0( \shM(m)  )-h^1( \shM(m))+h^2( \shM(m)) \\
  & = &  \frac{1}{2}  (mH+D)^2 - \frac{1}{2}  (mH+D). K_S  + \chi(\O_S) \\
 & = &  \frac{-4}{2}  +0+ 2 
 \\ & = & 0  \, .
\end{eqnarray*}
Because $\shM(0)$, $\shM(-1)$ are neither ample nor antiample they have neither nonzero sections nor second cohomology.
Therefore also the first cohomology vanishes.
\end{proof}

\section{Interpretation as local-cohomological Hilbert-Kunz multiplicities}
\label{interpretation}

Let $Y= \Proj R$, where $R$ is a standard-graded ring, endowed with $\O_Y(1)$.
For a  quasicoherent sheaf $\shM$ on $Y$ one sets $\shM(m)= \shM \otimes_{\O_Y} \O_Y(m)$ and
\[ M:=  \Gamma_*(\shM) = \bigoplus_{m \in \Z} \Gamma(Y, \shM (m)) \, , \]
which is a graded $R$-module. The $\O_Y(m)$ are invertible sheaves on $Y$ related by
$\O_Y(m) \otimes_{\O_Y} \O_Y(m') \cong \O_Y(m+m')$,
and they are free on $D_+(x)$ for any linear polynomial $x \in R_1$.
 The tensor multiplication induces compatible actions
 $ \O_Y(\ell) \otimes_{\O_Y} \shM(m) \rightarrow \shM(\ell  + m)$.

We relate this to the pull-back $\pi^*(\shM)$ for the cone mapping
$\pi:U \rightarrow Y$, where $U=D(R_+) \subseteq \Spec R=X$ is the punctured spectrum.
The multiplicative group acts on $U$
(this corresponds to the grading by \cite[Proposition 4.7.3]{sga3})
and the quotient is the projective variety $Y$.
For a homogeneous element $h \in R$ of positive degree the cone mapping restricts to the affine morphism
$D(h) \rightarrow D_+(h)$ corresponding to the ring homomorphism
\[ \Gamma(D_+(h), \O_Y) =(R_h)_0    \longrightarrow     R_h \cong \Gamma(D(h), \O_X)   \,  \]
and we have
\[ \Gamma(D(h), \pi^* \O_Y)  \cong R_h \cong \bigoplus_{m \in \Z} (R_h)_m \cong \bigoplus_{m \in \Z} \Gamma(D_+(h) ,  \O_Y(m)  ) \, .\]
If $h$ has degree one, then $R_h$ has units in degree one and then $R_h \cong (R_h)_0 [T,T^{-1}]$.
These local isomorphisms together yield a global isomorphism
\[  \Gamma(U, \pi^* \O_Y)  =  \bigoplus_{m \in \Z} \Gamma(Y, \O_Y(m))   \,  \]
(which is $R$ if $R$ is normal of dimension $ \geq 2$).
Locally we have isomorphisms
\begin{eqnarray*}
 \Gamma (D(h), \pi^* \shM) &   \cong  & R_h \otimes_{(R_h)_0 }   \Gamma( D_+(h) , \shM)      \\
& \cong &   \bigoplus_{m \in \Z}     (R_h)_m \otimes_{(R_h)_0 }   \Gamma( D_+(h) , \shM) \\
& \cong &   \bigoplus_{m \in \Z}      \Gamma( D_+(h) , \O_Y(m) \otimes_{\O_Y}  \shM) \\
& \cong & \bigoplus_{m \in \Z} \Gamma(D_+(h), \shM(m) )
\end{eqnarray*}
and these isomorphisms are compatible with the action of
$ \Gamma(U, \pi^* \O_Y)  =  \bigoplus_{\ell \in \Z}\! \Gamma(Y, \O_Y(\ell))$
given by the above mentioned action of
$         \Gamma ( D_+(h), \O_Y( \ell) ) \!  =   (R_h)_\ell     $ on $ \Gamma(D_+(h), \shM (m) )$
with values in  $ \Gamma(D_+(h), \shM (\ell + m) )$.
The $\O_U$-module $\pi^*\shM$ is hence a $\Z$-graded $R$-module and the grading is locally given by
$\bigoplus_{m \in \Z} \Gamma(D_+(h), \shM (m) )$.
Therefore we get a graded isomorphism
\[\Gamma(U, \pi^* \shM ) = \bigoplus_{m \in \Z} \Gamma(Y, \shM(m)) =M \, . \] 
Th graded $\check{\rm C}$ech complex of $M$ restricted to $U$
(for a cover given by homogeneous elements)
is just the direct sum over all $m \in \Z$
of the $\check{\rm C}$ech complexes
for $\Gamma (Y, \shM(m))$. In particular, the sheaf cohomology of $M$ over $U$
is the direct sum over $m \in \Z$ of the sheaf cohomologies of $\shM(m)$
on $Y$. Hence we can translate the previous results to sheaf cohomology on the punctured spectrum and to local cohomology
(see also \cite[Theorem 20.4.4]{brodmannsharp}).

\begin{cor}
\label{module2irrational}
There exists a three-dimensional hypersurface domain $R=K[X,Y,Z,W]/(F)$ where $F$ is homogeneous of degree four and where $K$
has positive characteristic $p \gg 0$ 
and an $R$-module $M$ of rank one which is invertible on the punctured spectrum such that the limit
(the second local cohomological Hilbert-Kunz multiplicity)
\[  \operatorname{HK}^2 (M)  =   \lim_{e \rightarrow \infty}    \frac{ \lg   \left(  H^2_m (F^{e*} M  )  \right) }{p^{3e} }\]
is an irrational number. Moreover, this number is independent of $p$.
\end{cor}
\begin{proof}
We take $F=\operatorname{det} (A)$,
where $A$ is a $4\times 4$-matrix with linear entries as in Section \ref{determinantalquartic},
defined over $\Z$ and such that $\surface= V_+(F)$
is smooth in characteristic zero with Picard rank two. Let $p$ be a prime number as in Lemma \ref{ampleeffectivesublattice}.
We look at the invertible sheaf $\shM$ from Corollary \ref{oguisoirrational} and set
\[ M:=\Gamma_*(\shM)=\bigoplus_{m \in \Z} \Gamma (\surface, \shM(m)) =\Gamma(U, \pi^*\shM) \, .\]
This is a graded $R$-module, bounded from below, whose sheafification restricted to $U$ is $M|_U = \pi^*\shM $
by \cite[Lemma II 5.14]{hartshorne} (or \cite[Proposition 5.1.2]{ega2}). In particular,
the restriction of $M$ to $U$ is invertible and has rank one.
We have
$H^1(U, \pi^*\shM ) = \bigoplus_{m \in \Z} H^1(\surface , \shM(m))$ and more generally
\[ H^1(U,  \pi^* (F^{e*} \shM ))
= \bigoplus_{m \in \Z} H^1(\surface , \shM^{q} (m)) \, , \] and the sum on the right is finite for each $e$.
For the Frobenius pull-backs we have
\[ (F^{e*}M)|_U = F^{e*}(M|_U) = F^{e*} (\pi^* \shM) = \pi^* (F^{e*} \shM ) = \pi^*(\shM^q) \, . \]
Hence
\[ \dim_K  H^1(U, F^{e*}  M  ) = \sum_{m \in \Z}  \dim_K H^1(\surface , \shM^q (m)) \]
and divided by $q^3$ these numbers have an irrational limit. 
Finally, we have $H^1(U,F^{e*}M) \cong  H^2_\m (F^{e*}M)$ by the exact sequence relating local and global cohomology.
The independence of $p$ follows from the construction.
\end{proof}

\begin{cor}
\label{module1irrational}
There exists a three-dimensional hypersurface domain $R=K[X,Y,Z,W]/(F)$
where $F$ is homogeneous of degree four and $K$
has positive characteristic $p \gg 0 $,
a finitely generated $R$-module $Q$ and a submodule $N \subseteq Q$ with $Q/N$ of finite length
such that the limit
\[  \lim_{e \rightarrow \infty}    \frac{ \lg   \left(  H^0 (U, F^{e*} Q  )/ N^{[q]}  \right) }{p^{3e} }\]
is an irrational number.
\end{cor}
\begin{proof}
We start again with $\shM$ from Corollary \ref{oguisoirrational} and let
\[ 0 \longrightarrow \shM \longrightarrow \bigoplus_{j = 0}^s \O_\surface (-\alpha_j) \longrightarrow \shQ \longrightarrow 0 \]
be exact on $\surface$
(which we get by resolving the dual of $\shM$).
Let $Q= \bigoplus_{m \in \Z} \Gamma(\surface, \shQ(m))$.
Note that $Q$ is not the quotient of $\bigoplus_{j=1}^s R(-\alpha_j)$ modulo
$M = \bigoplus_{m \in \Z} \Gamma (\surface, \shM(m))$; however, we have a complex
\[0 \longrightarrow M \longrightarrow \bigoplus_{j=1}^s R(-\alpha_j) \longrightarrow Q \longrightarrow 0\]
whose restriction to $U$ is exact. The Frobenius pull-back of this complex is again a complex, and its restriction to $U$
is exact. Let $N$ denote the image inside $Q$.
Then we have on $U$ a long exact graded cohomology sequence 
\[ \bigoplus_{j=1}^s R(- \alpha_j q) = H^0(U,  \bigoplus_{j=1}^s  \O_U(-  \alpha_j q)  )
\longrightarrow     H^0(U, F^{e*} Q)  \]
\[    \longrightarrow      H^1(U, F^{e*}  M   )      \longrightarrow 0= H^1(U, \bigoplus_{j=1}^s  \O_U(-   \alpha_j q) ) \, . \]
Therefore
\[     H^0(U, F^{e*} Q)    / \im  (  N^{[q]}     )         \cong H^1(U, F^{e*} M )= H^2_\m ( F^{e*} M)  \, \]
and this gives the irrational limit by Corollary \ref{module2irrational}.
\end{proof}

\begin{cor}
There exists a three-dimensional hypersurface domain $R=K[X,Y,Z,W]/(F)$
where $F$ is homogeneous of degree four and $K$
has positive characteristic $p \gg 0$,
a finitely generated $R$-module $Q$ and a submodule $N \subseteq Q$
with $Q/N$ of finite length and such that at least one of the following three limits
\[    \lim_{e \rightarrow \infty} \!   \frac{ \lg   \left( H^0_\m (  ( F^{e*} Q ) / H^0_\m ( N^{[q]})   \right) }{p^{3e} }, 
\lim_{e \rightarrow \infty} \!   \frac{ \lg   \left(   ( F^{e*} Q  )/N^{[q]}   \right) }{p^{3e} },
\lim_{e \rightarrow \infty}  \!  \frac{ \lg   \left(  H^1_\m ( F^{e*} Q  )  \right) }{p^{3e} }, \]
is an irrational number.
\end{cor}
\begin{proof}
This follows from Corollary \ref{module1irrational} and the short exact sequence of Lemma \ref{relatingsequence}.
\end{proof}

For the following corollary I am grateful to H. Dao,
who brought \cite[Theorem 2.9]{daolimiller} and \cite[Theorem 4.15]{daosmirnov} to my attention.

\begin{cor}
\label{module0irrational}
There exists a three-dimensional hypersurface domain
$R=K[X,Y,Z,W]/(F)$ where $F$ is homogeneous of degree four and where $K$
has positive characteristic $p \gg 0$ and an ideal $I \subset R$ 
such that the limit
(the zeroth local cohomological Hilbert-Kunz multiplicity)
\[  HK^0 (R/I) =  \lim_{e \rightarrow \infty}    \frac{ \lg   \left( H^0_\m (  F^{e*} (R/ I )   \right) }{p^{3e} }   
=    \lim_{e \rightarrow \infty}    \frac{ \lg   \left( H^0_\m (    R/ I^{[q]} )   \right) }{p^{3e} }        \]
is an irrational number. This number is independent of $p$. 
\end{cor}
\begin{proof}
We start again with the invertible sheaf $\shM$ on $\surface$ as in Corollary \ref{oguisoirrational}
and we look at a twist such that $\shM(\ell) \subseteq \O_\surface$ becomes an ideal sheaf on $\surface$.
Then the shifted version $J=M(\ell)$ is an ideal of $R$ which is isomorphic to $ M$,
in particular it is an invertible ideal sheaf on $U$.
Note that restricted to $U$, the Frobenius powers, the ordinary powers and the symbolic powers of this ideal are the same.
Under the given assumption we can write  $J=(a:b)$ for some nonzero elements $a,b \in R$, see \cite[Theorem 2.9]{daolimiller}.
We have a short exact sequence of $R$-modules
\[0 \longrightarrow R/(a^q:b^q) \longrightarrow R/(a^q) \longrightarrow R/(a^q,b^q)  \longrightarrow 0 \, ,\]
where the map on the left is given by $1 \mapsto b^q$. For the $q$th symbolic power we have $J^{(q)} =(a^q:b^q) $. We set $I=(a,b)$.
From $H^0_\m (R/(a^q)) =H^1_\m (R/(a^q)) =0$, which holds because the depth of $R/(a^q)$ is two, we infer
$H^0_\m (R/(I^{[q]} ))  \cong H^1_\m(R/ J^{(q)} ) $
and from
$0 \rightarrow J^{(q)} \rightarrow R \rightarrow R/ J^{(q)} \rightarrow 0 $
we conclude $H^1_\m(R/ J^{(q)}) \cong H^2_\m (J^{(q)})$, since $R$ is Cohen-Macaulay of dimension three.
Finally we get
\[H^2_\m(J^{(q)})  =H^1(U, J^{(q)}) = H^1 (U, F^{e*}  J )  =H^2_\m ( F^{e*}   J)  \, , \]
where $U$ denotes the punctured spectrum and where the equation in the middle holds since $U$ is regular and $J$ is invertible on $U$.
So we get altogether $H^0_\m (  F^{e*} ( R/ I )  )  \cong H^2_\m( F^{e*}  J) $
and the limit behavior follows from Corollary \ref{module2irrational}.
\end{proof}

In the final result of this section we get away from local cohomological versions of Hilbert-Kunz multiplicities
and return to the artinian case with the help of  
\cite[Theorem 5.2]{daosmirnov}.

\begin{thm}
\label{moduleartinianirrational}
There exists a three-dimensional hypersurface domain $R=K[X,Y,Z,W]/(F)$ where $F$ is homogeneous of degree four and where $K$
has positive characteristic $p \gg 0$ and a finitely generated artinian $R$-module $M$
such that the limit
(the Hilbert-Kunz multiplicity)
\[  HK (M) =  \lim_{e \rightarrow \infty}    \frac{ \lg   \left(  F^{e*} M  \right) }{p^{3e} }  \]
is an irrational number. 
\end{thm}
\begin{proof}
By Corollary \ref{module0irrational} we know that there is an ideal $I =(a,b) $ such that
$ \lim_{e \rightarrow \infty} \frac{ \lg \left( H^0_\m (F^{e*} (R/I) ) \right)}{p^{e3}} $ is irrational. Moreover, $J=(a:b)$
is the graded ideal coming from a shift of the invertible sheaf
$\shM$ on a smooth determinantal quartic surface $S$ as described in Section \ref{determinantalquartic}. 
By Lemma \ref{oguisoh1}, this $\shM$ which we have picked to define $J$ and all its shifts $\shM(m)$ do not have first cohomology
on $\surface$.
From this we deduce with the short exact sequences from Corollary \ref{module0irrational} (for $q=1$)
\[H^0_\m( R/I)
= H^1_\m (R/J) = H^2_\m (J )  = H^1  (U, J) =0 \, .   \]
Therefore $R/I$ has depth at least one and we can apply \cite[Theorem 5.2]{daosmirnov} which states that for a hypersurface ring $R$
and a finitely
generated $R$-module $N$ of depth $\geq 1$ there exist two $R$-modules $N_1$ and $N_2$ of finite length such that
\[ \operatorname{HKF}^0 (N, e )   =  \frac{ \operatorname{HKF} (N_1 , e ) -  \operatorname{HKF} (N_2 , e ) }{2}   \, .   \]
The same equality holds by dividing through $p^{3e}$ and going to the limit, i.e. for the Hilbert-Kunz multiplicities, so
\[  \operatorname{HK}^0 (N)   =  \frac{ \operatorname{HK} (N_1) -  \operatorname{HK} (N_2 ) }{2}   \, .   \]
Since we have an irrational number on the left for $N=R/I$,
at least one of the numbers on the right must be irrational as well.
\end{proof}

\section{Reductions to the case of a maximal ideal}
\label{reductions}

In this section we show that there also exists a local noetherian ring with irrational Hilbert-Kunz multiplicity.
This is achieved by two algebraic reductions which are independent of previous results. First we show that if
a local noetherian ring $R$ of positive characteristic has an $R$-module of finite length with irrational Hilbert-Kunz multiplicity,
then there exists also a primary ideal in some polynomial ring over $R$ with irrational Hilbert-Kunz multiplicity. Then we show that whenever we have
a local ring with a primary ideal with irrational Hilbert-Kunz multiplicity,
then we can also construct a local ring where the maximal ideal has
irrational Hilbert-Kunz multiplicity.

\begin{thm}
\label{reductionideal}
Let $(R,\m)$ be a local noetherian ring containing a field of positive characteristic $p$.
Suppose that there exists an $R$-module $M$ of finite length with irrational
Hilbert-Kunz multiplicity.
Then there exists an ideal $I$ in some polynomial ring $R[T_1, \ldots, T_m]$,
primary to $\m +(T_1, \ldots, T_m)$, such that its Hilbert-Kunz multiplicity is irrational.
\end{thm}
\begin{proof}
Let ${\mathfrak a} = \operatorname{Ann}_R M$ be the annihilator of $M$, which is an $\m$-primary ideal in $R$.
If its Hilbert-Kunz multiplicity is irrational, then we are already done, so assume it is a rational number.
Let
\[  R^n \stackrel{A}{\longrightarrow} R^m \longrightarrow M \longrightarrow 0 \]
be a presentation of $M$, where $A$ is an $m \times n$-matrix. Let 
\[c_j= \begin{pmatrix} f_{1j} \\ \vdots \\ f_{mj} \end{pmatrix} ,\, j=1, \ldots , n,\] 
be the columns of $A$ (each with $m$ entries). Note that the elements $he_i$, $1 \leq i \leq m$, $h \in {\mathfrak a}$
($e_i$ being the standard basis of $R^m$),
lie in the image of $A$, i.e. there exist linear combinations $he_i = \sum_{j=1}^n r_j c_j$ of the columns $c_j$.
Raising this system of equations to the $q$th power, we see that
$h^q e_i $ is a linear combination of the $q$th power of the columns (taken componentwise).
In particular ${\mathfrak a}^{[q]} $ annihilates $F^{e*}M$.

Consider in the polynomial ring $S=R[T_1, \ldots, T_m]$ the ideal
\[ I = {\mathfrak a} + ( f_{1j}T_1 + \cdots +f_{mj}T_m ,\, j=1, \ldots , n)   +( T_iT_j,\, 1 \leq i,j \leq m) \, .\]
So each column contributes with a linear polynomial.
This ideal is homogeneous in the standard grading (we will work soon with different gradings as well) and clearly
$\m +(T_1, \ldots , T_m)$-primary.
Its $q$th Frobenius power is
\[ I^{[q]} = {\mathfrak a}^{[q]} + ( f_{1j}^qT^q_1 + \cdots +f_{mj}^qT^q_m ,\, j=1, \ldots , n)  +( T^q_iT^q_j,\, 1 \leq i,j \leq m) \, .\]
We have to compute the lengths of $S/I^{[q]}$. For this we work with the $(\Z/(q))^m$-grading (over $R$)
of the polynomial ring $S$, i.e.
each $T_i$ gets the degree $e_i \in  (\Z/(q) )^m$
($e_i$ being now the standard basis in this group).
Then the ideal $I^{[q]}$ is homogeneous in this grading and in fact generated by
elements of degree $0$. Therefore the residue class ring $S/I^{[q]}$ is  $(\Z/(q) )^m$-graded as well.
Hence we can compute its length by computing the lengths of its graded
pieces (which are $R$-modules). These pieces are indexed by $\mu \in  (\Z/(q) )^m$ (where $0 \leq \mu_i < q$ for each $i$)
and its monomial representatives in $S_\mu$ are
\[ (1)\,\,\, T^\mu,\, \,\,
(2) \,\,\, T^\mu T_1^q, \ldots , T^\mu T_m^q,\,  \,\,
(3)\,\,\,  T^\mu T_i^qT_j^q ,\, 1\leq i,j \leq m,\ \rm{etc.} \,  , \]
so these monomials form an $R$-module generating system of $(S/I^{[q]} )_\mu$.
We have to understand what happens to these monomials 
modulo $I^{[q]}$. 
For type (1) we only have to consider the ideal ${\mathfrak a}^{[q]}$. For type (2), note that by the remark made above, each
${\mathfrak a}^{[q]} T^\mu T_i^q$ is contained in the $R$-module generated by
$T^\mu(\sum_{i=1}^m f^q_{ij}T_i^q) $, $j=1, \ldots ,n $,
so only these $q$-powers of the linear polynomials are relevant,
and type (3) is completely killed by $I^{[q]}$.
Therefore we get an $R$-module isomorphism
\[ (S/I^{[q]})_\mu  \cong 
R/ {\mathfrak a}^{[q]}   \oplus
\left(   RT^\mu T_1^q  \oplus \cdots \oplus RT^\mu T_m^q   \right)/(  T^\mu( \sum_{i=1}^m f^q_{ij} T_i^q ),\, j=1,\ldots, n)    \]
and this $R$-module is isomorphic to
\[ R/{\mathfrak a}^{[q]}  \oplus F^{e*}M \,. \]
Hence its length is $\operatorname{HKF}^R (R/{\mathfrak a}, e) +  \operatorname{HKF}^R (M, e)    $.
Since there exist $q^m$ graded pieces, we get
\[\operatorname{HKF}^S (S/I,e)  = q^m (  \operatorname{HKF}^R (R/{\mathfrak a}, e) +  \operatorname{HKF}^R (M, e)    )  \]
and therefore
\[\frac{ \operatorname{HKF}^S (S/I,e)}{q^{\dim S} }
= \frac{  \operatorname{HKF}^R (R/{\mathfrak a}, e) +  \operatorname{HKF}^R (M, e)    ) }{q^{\dim R} } \, .\]
Taking the limit for $e \rightarrow \infty$
we get
\[  \operatorname{HK}^S (S/I)  =   \operatorname{HK}^R (R/{\mathfrak a}) +  \operatorname{HK}^R (M)   \, .     \]
Since $ \operatorname{HK}^R (R/{\mathfrak a})$ is assumed to be rational and $\operatorname{HK}^R (M)$ is  irrational, also
$ \operatorname{HK}^S (S/I) $ is irrational.
\end{proof}

\begin{thm}
\label{reductionmaximal}
Let $(R, \m)$ be a local noetherian normal excellent domain containing a field $K$
of positive characteristic and let $I$ be an $\m$-primary ideal.
Then there
exists a local noetherian domain
$(S,\n)$ such that the Hilbert-Kunz multiplicity of $S$ is the Hilbert-Kunz multiplicity of $I$
up to a multiple of a rational number.
\end{thm}
\begin{proof}
Let $I=(f_1, \ldots , f_n)$ and consider the subalgebra $T'=K[f_1, \ldots, f_n] \subseteq R$.
Let 
$ T =  T'_{\m'} $ be the localization, where $\m' =  \m \cap T' = (f_1, \ldots, f_n) $.
Then the extended ideal is $\m' R = I$. 

Now we work with the completions of $T$ and of $R$. Note that the Hilbert-Kunz multiplicities are not changed by completions.
The completion $\hat{R}$ of $R$ is again a domain by the normality and excellence assumption (see \cite[Corollaire 7.6.2]{ega4.2}),
and we get a ring homomorphism
$\hat{T} \rightarrow \hat{R}$.
Let ${\mathfrak p}$ be its kernel and let $S= \hat{T}/{\mathfrak p}$.
Then $S \subseteq \hat{R}$
is a finite extension of domains by \cite[Theorem 30.6]{nagata} and the maximal ideal $\n$ of $S$ extends to $I \hat{R}$. 
By \cite[Theorem 2.7]{watanabeyoshidacolength}
we have $\operatorname{HK} (\n) = \frac{ \operatorname{HK} (I \hat{R} )  [\hat{R}/\m \hat{R} : S/ \n] }{   [Q(\hat{R}) :Q(S)]}$.
\end{proof}

\begin{thm}
\label{hilbertkunzirrational}
There exists a local noetherian domain whose Hilbert-Kunz multiplicity is an irrational number.
\end{thm}
\begin{proof}
This follows from Theorem \ref{moduleartinianirrational}, Theorem \ref{reductionideal}
followed by localization at the irrelevant ideal
and
Theorem \ref{reductionmaximal}. 
\end{proof}


\bibliographystyle{amsalpha}
\bibliography{bibfile}

\end{document}